\documentclass[11pt]{article}
\usepackage{graphicx}
\usepackage{amsmath,amsthm,amssymb}
\usepackage{amscd}
\newtheorem{thm}{Theorem}[section]
\newtheorem{theorem}[thm]{Theorem}
\newtheorem{proposition}[thm]{Proposition}
\newtheorem{lemma}[thm]{Lemma}

\newtheorem{claim}[thm]{Claim}

\title{Count Matroids of Group-Labeled Graphs\footnote{Part of this research was conducted when the first author was a graduate student at the University of Tokyo.}}
\author{Rintaro Ikeshita\thanks{e-mail: {\tt  ikeshita.rintaro@gmail.com}} \and Shin-ichi Tanigawa\thanks{Research Institute for Mathematical Sciences,
Kyoto University, Sakyo-ku, Kyoto 606-8502, Japan, and
Centrum Wiskunde $\&$ Informatica (CWI), Postbus 94079, 1090 GB
Amsterdam, The Netherlands.
e-mail: {\tt tanigawa@kurims.kyoto-u.ac.jp}}}
\begin{document}
\maketitle

\begin{abstract}
A graph $G=(V,E)$ is called $(k,\ell)$-sparse if $|F|\leq k|V(F)|-\ell$ for any nonempty $F\subseteq E$,
where $V(F)$ denotes the set of vertices incident to $F$.
It is known that the family of the edge sets of $(k,\ell)$-sparse subgraphs  forms the family of independent sets of a matroid, called the $(k,\ell)$-count matroid of $G$.
In this paper we shall investigate lifts of the $(k,\ell)$-count matroids by using group labelings on the edge set. By introducing a new notion called near-balancedness, we shall identify a new class of matroids whose independence condition is described as a count condition of the form
$|F|\leq k|V(F)|-\ell +\alpha_{\psi}(F)$ for   some function $\alpha_{\psi}$   determined by a given group labeling $\psi$ on $E$.
\end{abstract}

%%%%%%%%%%%%%%%%%%%%%%%%%%%%%%%%%%%%%%%%%%%%%%%%%%%%%%%%
% Keywords (3 $\sim$ 5 words)
%%%%%%%%%%%%%%%%%%%%%%%%%%%%%%%%%%%%%%%%%%%%%%%%%%%%%%%%
\begin{quote}
\end{quote}
%\vspace{5mm}

%%%%%%%%%%%%%%%%%%%%%%%%%%%%%%%%%%%%%%%%%%%%%%%%%%%%%%%%
% Text
%%%%%%%%%%%%%%%%%%%%%%%%%%%%%%%%%%%%%%%%%%%%%%%%%%%%%%%%
\section{Count Matroids}
A {\em $\Gamma$-labeled graph} $(G,\psi)$ is a pair of a directed graph $G=(V,E)$
and an assignment $\psi$ of an element of a group $\Gamma$ with each oriented edge.
%such that reversing  the direction inverts the assigned element.
Although $G$ is directed, its orientation is used only for the reference of the gains, and we are free to change the orientation of each edge  by imposing the property on that if an edge has a label $g$ in one direction, then it has $g^{-1}$ in the other direction. Therefore we often do not distinguish between $G$ and the underlying undirected graph.
%Group-labeled graphs are also called {\em gain graphs}.
By using the group-labeling one can define variants of graphic matroids.
Among such variants, {\em Dowling geometries}~\cite{dowling1973class}, 
or their restrictions,  {\em frame matroids}~\cite{zaslavsky1991biased, zaslavsky1994}, 
are of most importance in the theory of matroid representations.
%Among such variants the most fundamental  one are {\em Dowling geometries}~\cite{dowling1973class}, 
%or their restrictions,  {\em frame matroids}~\cite{zaslavsky1991biased, zaslavsky1994}. 
In the frame matroid of $(G,\psi)$, an edge set $I$ is independent if and only if each connected component of $I$ contains no cycle or just one cycle which is {\em unbalanced}, i.e., the total gain through the cycle is not equal to the identity. 
By extending the notion of balancedness to any edge subsets such that 
$F\subseteq E$ is {\em unbalanced} (resp.~{\em balanced}) 
if it contains (resp.~does not contain) an unbalanced cycle, 
 the independence condition in the frame matroid can be equivalently written as 
\begin{equation}
\label{eq:1}
|F|\leq |V(F)|-1+\begin{cases}
0 & \text{ if $F$ is balanced} \\
1 & \text{ otherwise}
\end{cases} \qquad (\emptyset \neq F\subseteq I),
\end{equation}
where $V(F)$ denotes the set of vertices incident to $F$.
Notice that, if we ignore the last term, this condition is nothing but the independence condition in the graphic matroid of $G$, and hence the count condition exhibits how the graphic matroid is lifted (see \cite{whittle1989generalisation} for a discussion based on submodular functions).

There is a natural generalization of the count condition for cycle-freeness, known as {\em $(k,\ell)$-sparsity}. 
We say that an edge set $I$ is {\em $(k,\ell)$-sparse} if $|F|\leq k|V(F)|-\ell$ holds for any nonempty $F\subseteq I$. It is known that the set of $(k,\ell)$-sparse edge sets in $G$ forms a matroid on $E$, 
called the {\em $(k,\ell)$-count matroid} of $G$.
For $k\geq \ell$, the $(k,\ell)$-count matroids  appear in several contexts in graph theory and combinatorial optimization as they are the unions of copies of the graphic matroid and the bicircular matroid (see, e.g., \cite{Frank2011}),  
and in particular the $(k,k)$-sparsity condition is Nash-Williams' condition for a graph to be decomposed into $k$ edge-disjoint forests.
The $(k,\ell)$-count matroids  appear in rigidity theory and scene analysis for various kinds of pairs of $k$ and $\ell$ (see, e.g., \cite{Whitley:1997}).

Since the  $(1,1)$-count matroid coincides with the graphic matroid, it is natural to ask when a count condition of the form 
\begin{equation}
\label{eq:count_intro}
|F|\leq k|V(F)|-\ell+\alpha_{\psi}(F) \qquad (\emptyset\neq F\subseteq I)
\end{equation}
for some function $\alpha_{\psi}$ determined by the group labeling induces a matroid of $(G,\psi)$. 
In this paper we shall establish a general construction of $\alpha_{\psi}$ for which the count condition induces a matroid. Our work is in fact motivated from characterizations of the rigidity of graphs with symmetry. Recent works on this subject reveal connections of the infinitesimal rigidity of symmetric bar-joint frameworks with count conditions of the form (\ref{eq:count_intro}) on the quotient group-labeled graphs~\cite{mt13,mt14,ross2011, t,jkt, ns}, where each symmetry and each rigidity model gives a distinct $\alpha_{\psi}$.
In Section~2 we give examples, several of which were not known to form matroids before. 
In this context it is crucial to know whether a necessary count condition forms a matroid or not (see, e.g., \cite{mt13, mt14, t, jkt, ns}).

%To explain count conditions appeared in rigidity theory, we need to use structures of the underlying group more. 
Our construction uses more refined properties of group-labelings than balancedness.
To explain this we need to introduce some notation.
Let $(G,\psi)$ be a $\Gamma$-labeled graph.
The set of nonempty connected edge sets in $G$ is denoted by ${\cal C}(G)$.
A {\em walk} in $G$ is a sequence $W=v_0, e_1, v_1, e_2, \dots, e_k, v_k$ of vertices and edges such that 
$v_{i-1}$ and $v_i$ are the endvertices of $e_i$ for every $1\leq i\leq k$.
The {\em gain} $\psi(W)$ of the walk $W$ is defined to be $\psi(e_1)^{\sigma(e_1)}\cdot \psi(e_2)^{\sigma(e_2)}\cdot \dots \cdot \psi(e_k)^{\sigma(e_k)}$,
where $\sigma(e)=1$ if $W$ traces $e$ in the forward direction and otherwise
$\sigma(e)=-1$.
For $F\in {\cal C}(G)$ and $v\in V(F)$ 
let $\langle F\rangle_{v,\psi}$ be the subgroup of $\Gamma$ generated by 
$\psi(W)$ for all closed walks $W$ starting at $v$ and using only edges in $F$.
%For any $\gamma\in \Gamma$, switching operations at all vertices with $\gamma$ results in an equivalent gain function $\psi'$ with $\psi'(e)=\gamma \psi(e)\gamma^{-1}$.
%Such a gain function is said to be {\em conjugate} of $\psi$.
It is known that $\langle F\rangle_{v, \psi}$ is conjugate to $\langle F\rangle_{u, \psi}$  for any $u, v \in V(F)$ (see, e.g., \cite{jkt}).
Hence the conjugate class is uniquely determined for each $F\in {\cal C}(G)$, which is denoted by $[F]$. 

For a group $\Gamma$ and $S\subseteq \Gamma$, let $\langle S\rangle$ be the subgroup generated by elements in $S$ and let $[S]$ be the conjugate class of $\langle S\rangle$ in $\Gamma$.
Also the identity of $\Gamma$ is denoted by $1_{\Gamma}$.

We say that a function $\alpha:2^{\Gamma}\rightarrow \mathbb{Z}$ is {\em polymatroidal} if
\begin{description}
\item[(c1)] $\alpha(\emptyset)=0$,
\item[(c2)] $\alpha(X)+\alpha(Y)\geq \alpha(X\cup Y)+\alpha(X\cap Y)$ for any $X, Y\subseteq \Gamma$,
\item[(c3)] $\alpha(X)\leq \alpha(Y)$ for any $X\subseteq Y\subseteq \Gamma$,
\item[(c4)] $\alpha(\gamma X\gamma^{-1})=\alpha(X)$ for any $X\subseteq \Gamma$ and $\gamma \in \Gamma$,
\item[(c5)] $\alpha(\langle X\rangle)=\alpha(X)$ for any $X\subseteq \Gamma$.
\end{description}
Since $\alpha$ is closed under taking the closure and the conjugate, $\alpha$ induces a class function (i.e., a function on the conjugate classes), which is denoted by $\tilde{\alpha}$.
For $F\in {\cal C}(G)$ we often abbreviate $\tilde{\alpha}([F])$ by $\tilde{\alpha}(F)$.

%Such a polymatroidal function $\alpha$ is said to be {\em proper}.

The following was proved in \cite{t}.
\begin{thm}[Tanigawa~\cite{t}]
\label{thm:tan}
Let $(G,\psi)$ be a $\Gamma$-labeled graph, $\alpha:2^{\Gamma}\rightarrow \{0,1,\dots, k\}$ be a polymatroidal function.
Define $f_{\alpha}:{\cal C}(G)\rightarrow \mathbb{Z}$ by
\[
f_{\alpha}(F)=k|V(F)|-k+\tilde{\alpha}(F) \qquad (F\in {\cal C}(G)).
\]
Then the set ${\cal I}_{\alpha}(G)=\{I\subseteq E(G)\mid |F|\leq f_{\alpha}(F)\ \forall F\in {\cal C}(G)\cap 2^I \}$
forms the family of independent sets in a matroid.
\end{thm}

In this paper we shall extend Theorem~\ref{thm:tan}  for  general $\ell$. 
Interestingly,  replacing just ''$k|V(F)|-k$" with "$k|V(F)|-\ell$" in the definition of $f_{\alpha}$ may not produces a matroid  in general as shown in Example 3 in the next section, and our extension is achieved by introducing  a new notion, called near-balancedness.
Let $v$ be a vertex of $(G, \psi)$ and $\{E_1, E_2\}$ be a bipartition of the set of non-loop edges incident to $v$. 
If $v$ is not incident to a loop, 
then a {\em split} of $(G, \psi)$ (at a vertex $v$ with respect to a partition $\{E_1, E_2\}$) is 
defined to be a $\Gamma$-labeled graph $(G', \psi')$ obtained from $(G,\psi)$ by splitting $v$ into two vertices $v_1$ and $v_2$ such that $v_i$ is incident to all the edges in $E_i$ for $i = 1, 2$.
If $v$ is incident to a loop, then the split is defined to be a $\Gamma$-labeled graph $(G', \psi')$
obtained from $(G,\psi)$ by splitting $v$ into two vertices $v_1$ and $v_2$ such that $v_i$ is incident to the edges in $E_i$ for $i = 1, 2$, 
each balanced loop at $v$ is connected to $v_1$,
and each unbalanced loop at $v$ is regarded as an arc from $v_1$ to $v_2$, keeping the group-labeling\footnote{By definition of group-labeled graphs, the label of a loop is freely invertible. So, for an unbalanced loop $e$ at $v$ in $(G,\psi)$, the label of the new edge corresponding to $e$ in the split can be either $\psi(e)$ or $\psi(e)^{-1}$.},
where a loop is called {\em balanced} (resp., {\em unbalanced}) if its label is identity (resp., non-indentity).

We say that a connected set $F$ is  {\em near-balanced} if it is not balanced and there is a split of $(G,\psi)$ in which $F$ results in a balanced set.

\vspace{0.5\baselineskip}
\noindent
{\it Example 1.}
We give an example of near-balanced sets using Figure~\ref{fig:Z2}.
Let $e_1$ denote the edge from $v_2$ to $v_3$, and let $e_2$ and $e_3$ denote the edges from $v_1$ to $v_2$ with $\psi (e_2) = 1_\Gamma$ and $\psi (e_3) = g\neq 1_{\Gamma}$, respectively.
Consider $I_1=E(G)\setminus \{e_1\}$ and $I_2=E(G)\setminus \{e_2, e_3\}$ for example.
Then $I_1$ is not near-balanced since it contains two vertex-disjoint unbalanced cycles,
and $I_2$ is near-balanced since it is  balanced in a split of $(G,\psi)$ at $v_3$.
See Figure~\ref{fig:Z2}(d).
By the same reason $I_2\cup \{e_2\}$ is near-balanced.
On the other hand  the property of $I_2\cup \{e_3\}$ differs according to the order of $g$.
In fact $I_2\cup \{e_3\}$ is near-balanced  if and only if $g^2=1_{\Gamma}$.

\vspace{\baselineskip}

%For a symmetric polymatroidal function $\alpha:2^{\Gamma}\rightarrow \{0,1,\dots, \ell\}$,
%$F\in {\cal C}(G)$ is said to be {\em $(\alpha,[\Gamma'])$-critical} if 
%\begin{itemize}
%\item $F$ is near-$[\Gamma']$-balanced, and 
%\item $\tilde{\alpha}([F])>\tilde{\alpha}([\Gamma'])+k$.
%\end{itemize}

%We also require a certain Lipschitzness condition on  $\alpha$.
%We say that $\alpha:2^{\Gamma}\rightarrow \{0,1,\dots, \ell\}$ is {\em $k$-Lipschitz}
%if, for any subgroup $\Gamma'$ of $\Gamma$ and $g\in \Gamma\setminus \Gamma'$,
%\[
%g^2\in \Gamma' \quad \Rightarrow \quad \alpha(\Gamma'\cup \{g\})-\alpha(\Gamma')\leq k.
%\]

We also  remark that, for a polymatroidal function $\alpha:2^\Gamma \rightarrow \{0,1,\dots, \ell\}$, 
there is a unique maximum set $S\subseteq \Gamma$ with $\alpha(S)=0$ 
and $S$ actually forms a normal subgroup of $\Gamma$ due to the submodularity and the invariance under conjugation. 
Hence, taking the quotient of $\Gamma$ by $S$, throughout the paper we may assume  that 
\begin{description}
\item[(c6)] $\alpha(\{g\})\neq 0$ for any non-identity $g\in \Gamma$ and $\alpha(\{1_{\Gamma}\})=0$.
\end{description}
(The assumption for $\alpha(\{1_{\Gamma}\})$ can be achieved by adjusting $\ell$ in the following theorem.)
A polymatroidal function $\alpha$ is said to be {\em normalized} if 
it satisfies (c6).

Now we are ready to state  our main theorem for $\ell\leq k+1$.
The statement for $k$ and $\ell$ with $\ell\leq 2k-1$ is given in Section~\ref{sec:proof}.
\begin{thm}
\label{thm:main}
Let $k, \ell$ be integers with $k\geq 1$ and $0\leq \ell\leq k+1$, $(G,\psi)$ be a $\Gamma$-labeled graph, $\alpha:2^{\Gamma}\rightarrow \{0,1,\dots, \ell\}$ be a normalized polymatroidal function such that 
$\alpha(\Gamma')\leq k$ for any $\Gamma'\subseteq \Gamma$ with $\Gamma' \simeq \mathbb{Z}_2$. Define $f_{\alpha}:{\cal C}(G)\rightarrow \mathbb{Z}$ by
\[
f_{\alpha}(F)=k|V(F)|-\ell+
\begin{cases}
\min\{\tilde{\alpha}(F), k\} & (\text{if $F$ is near-balanced}) \\
\tilde{\alpha}(F) & (\text{otherwise}).
\end{cases}
\]
Then the set ${\cal I}_{\alpha}(G)=\{I\subseteq E(G)\mid |F|\leq f_{\alpha}(F)\ \forall F\in {\cal C}(G)\cap 2^I\}$
forms the family of independent sets in a matroid.
\end{thm}

%A couple of remarks  are listed below.
%First we should remark that $f_{\alpha}$ is well-defined 
%since if $F$ is $(\alpha, [\Gamma_1])$-critical and $(\alpha, [\Gamma_2])$-critical
%then $[\Gamma_1]=[\Gamma_2]$ (see Lemma~\ref{lem:near4}).
%
%As in the case of $(k,\ell)$-sparsity count, the condition $2k>\ell$ is necessary since otherwise ${\cal I}(G)=\emptyset$.

%\begin{figure}[ht]
%\centering
%\includegraphics[width = 4 cm]{./pic/Z2.pdf}
%\caption{
%aiueo
%}
%\end{figure}

%\medskip
%\noindent 
%{\it Remark 2.}
Examples given in the next section show the necessity of the lifting value condition for near-balanced sets and the value condition for $\alpha(\mathbb{Z}_2)$  in Theorem~\ref{thm:main}.

\begin{figure}[t]
\centering
\begin{minipage}{0.48\textwidth}
\centering
\includegraphics[width = 4 cm]{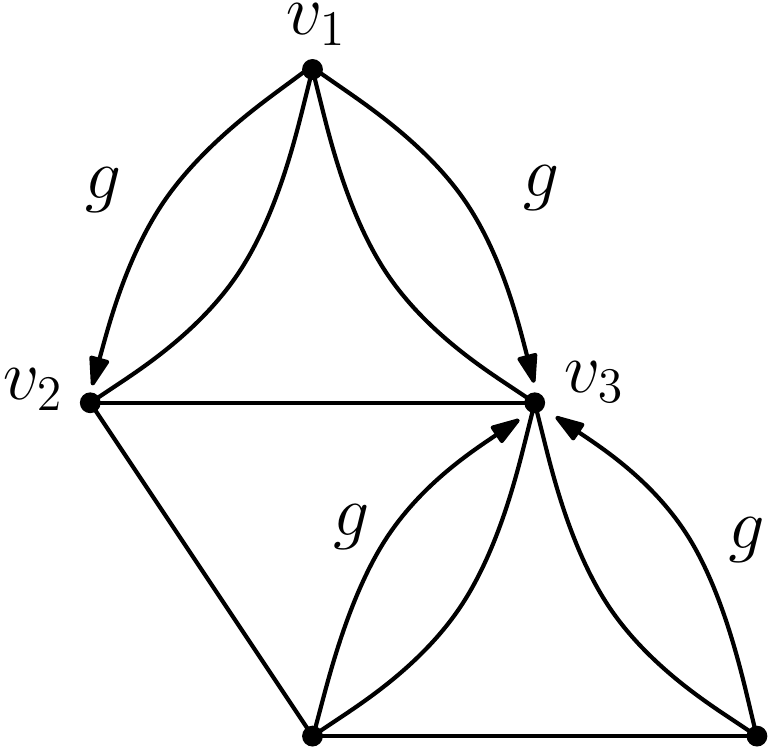}
\par
(a) %$(G,\psi)$.
\end{minipage}
\begin{minipage}{0.48\textwidth}
\centering
\includegraphics[width = 4 cm]{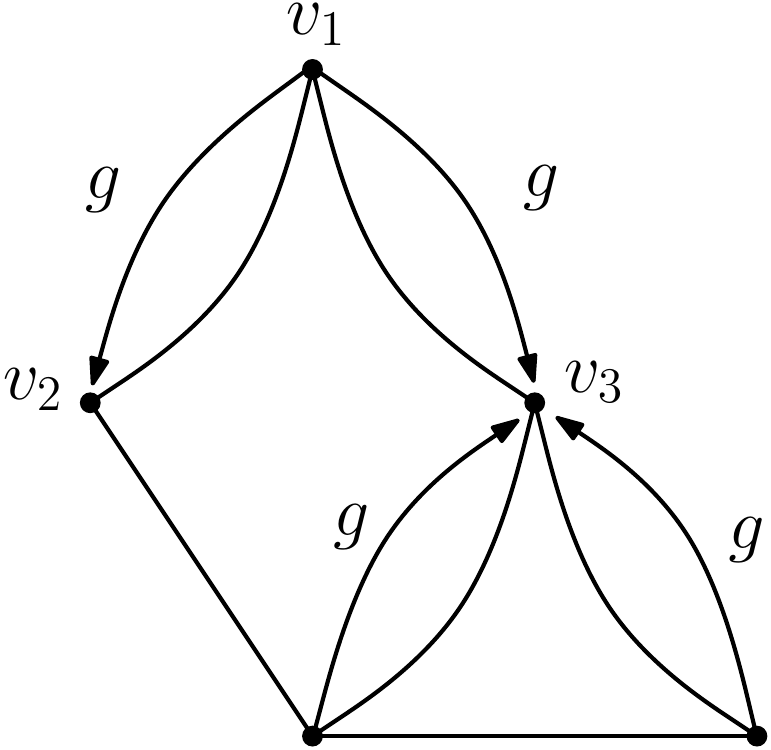}
\par
(b) %$I_1$.
\end{minipage}

\medskip

\begin{minipage}{0.48\textwidth}
\centering
\includegraphics[width = 4 cm]{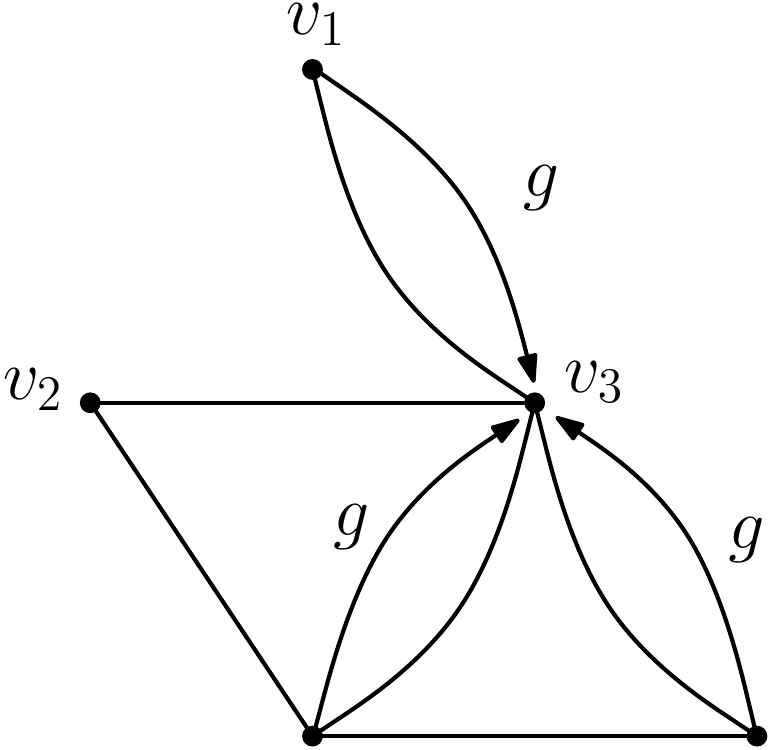}
\par
(c) %$I_2$.
\end{minipage}
\begin{minipage}{0.48\textwidth}
\centering
\includegraphics[width = 4 cm]{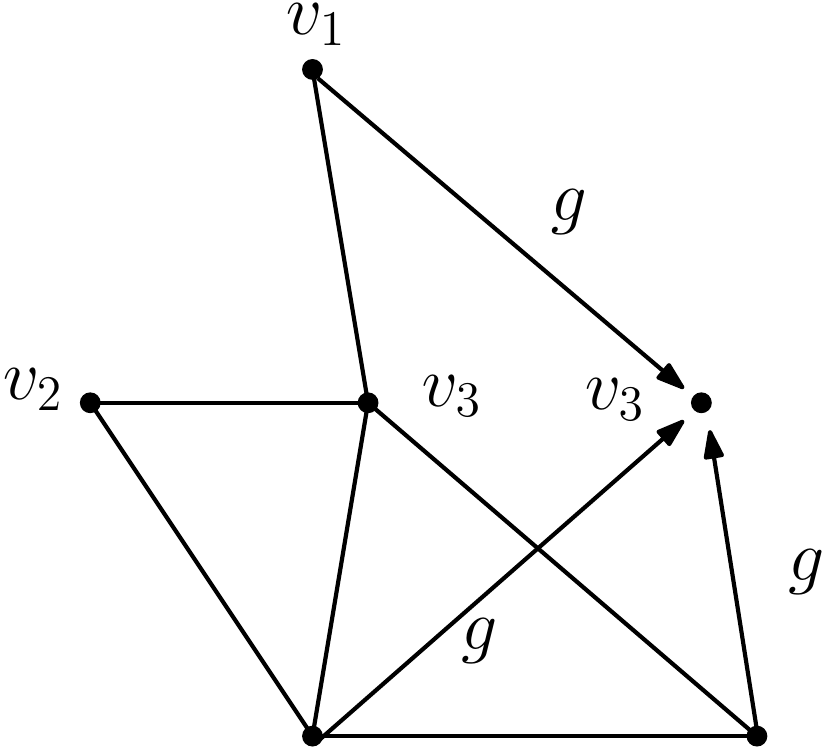}
\par
(d) %$I_2$ in a split of $(G,\psi)$ at $v_3$.
\end{minipage}
\caption{
(a) An example of a $\Gamma$-labeled graph $(G,\psi)$,
where $g \in \Gamma$ is not the identity and every non-labeled edge has the identity label $1_\Gamma$.
(b) A non near-balanced edge set $I_1$, (c) a near-balanced edge set $I_2$, and (d) $I_2$ in a split of $(G,\psi)$ at $v_3$.
%An example of a $\Gamma$-labeled graph $(G,\psi)$, which suggests the necessity of the assumption for $\alpha(\mathbb{Z}_2)$ in Theorem~\ref{thm:main}.
}
\label{fig:Z2}
\end{figure}

\section{Examples of Matroids}
Here we give examples of matroids given in Theorem~\ref{thm:main}.

\medskip
%In the following example, for $F\subseteq E$, $\langle F\rangle
\noindent {\it Example 2.}
The union of two copies of the frame matroid followed by Dilworth truncation results in a matroid whose independence condition is written by the following count:
\[
|F|\leq 2|V(F)|-3+\begin{cases}
0 & \text{if $F$ is balanced} \\
2 & \text{otherwise}
\end{cases} \qquad (F\in {\cal C}(G)).
\]
This is the case when $k=2, \ell=3$, and 
\[
\alpha(X)=\begin{cases}
0 & \text{$\langle X\rangle$ is trivial} \\
2 & \text{otherwise}
\end{cases} \qquad (X\subseteq \Gamma).
\]

\medskip

\noindent {\it Example 3.} \label{ex:2}
%A simplest example is the case when $k=\ell=1$. 
%According to the above discussion, the choice of a symmetric polymatroidal function  $\alpha$ is essentially unique according to the above discussion, i.e., $\alpha(X)=1$ if and only if $X\neq \{1_{\Gamma}\}$
%and $\alpha(1_{\Gamma})=0$.
%Hence 
%T and it is of the form
%\[
%\alpha(X)=\begin{
%Suppose that $\Gamma$ does not contain an element of order two.
In the context of graph rigidity, the following count condition appears as a necessary condition for the infinitesimal rigidity of symmetric bar-joint frameworks in the plane:
\[
|F|\leq 2|V(F)|-3+\begin{cases}
0 & \text{if $F$ is balanced} \\
3 & \text{otherwise}
\end{cases} \qquad (F\in {\cal C}(G)).
\]
The corresponding $\alpha$ is given by
\begin{equation*}
\alpha(X)=\begin{cases}
	0 & \text{ $\langle X\rangle$ is trivial} \\
	3 & \text{ otherwise}
\end{cases} \qquad (X\subseteq \Gamma).
\end{equation*}
Csaba Kir{\'a}ly pointed out that this condition does not induce a matroid in general. 
In Figure~\ref{fig:230} we give a smaller example for general groups.

\begin{figure}[t]
\centering
\includegraphics[width = 4 cm]{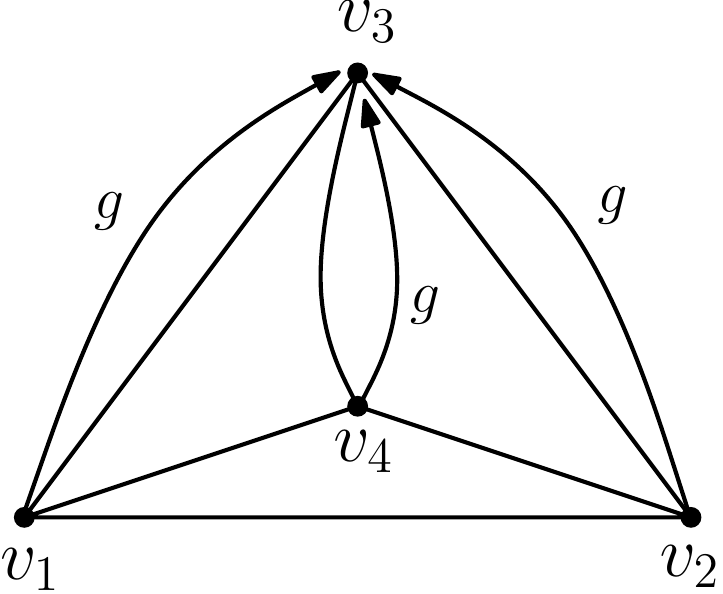}
\caption{
An example of a $\Gamma$-labeled graph $(G,\psi)$ not being a matroid in the count condition in Example~3, 
where $g \in \Gamma$ is not the identity and every non-labeled edge has label $1_\Gamma$.
Let $e_1$ denote the edge from $v_1$ to $v_2$ and $e_2$ and $e_3$ denote the edges from $v_1$ to $v_3$ with $\psi (e_2) = 1_\Gamma$ and $\psi (e_3) = g$, respectively.
Then $E_1 = E(G) \setminus \{ e_1 \}$ and $E_2 = E(G) \setminus \{ e_2, e_3 \}$ 
are maximal edge sets satisfying the count condition with distinct cardinalities.
Indeed, they are maximal because
%by counting, it can easily be checked that $E_1$ and $E_2$ satisfy the count condition.
%As for the maximality observe that 
$E_1 \cup \{ e_1 \}$ violates the $(2,0)$-sparsity while
each of $E_2 \cup \{ e_2 \}$ and $E_2 \cup \{ e_3 \}$ contains a balanced $K_4$, which indicates the violation of the $(2,3)$-sparsity for balanced sets.
}
\label{fig:230}
\end{figure}

Suppose that $\Gamma$ does not contain an element of order two.
Then Theorem~\ref{thm:main} implies that adding one additional condition for near-balanced sets gives rise to a matroid. 
Its independence condition is written as
\[
|F|\leq 2|V(F)|-3+\begin{cases}
0 & \text{if $F$ is balanced} \\
2 & \text{if $F$ is near-balanced} \\
3 & \text{otherwise} 
\end{cases} \quad (F\in {\cal C}(G)).
\]

This count condition still may not induce a matroid if $\Gamma$ contains an element of order two.
Consider the $\Gamma$-labeled graph in Figure~\ref{fig:Z2}, and define $I_1$ and $I_2$ as in Example 1. 
Suppose that $g^2=1_{\Gamma}$.
Then $I_1$ and $I_2$ are maximal sets in $\mathcal{I}_\alpha (G)$.
Indeed, by counting, it can easily be checked that $I_1, I_2 \in \mathcal{I}_\alpha (G)$.
As for the maximality of $I_2$,  observe that, for each $i = 2, 3$, $I_2 \cup \{ e_i \}$ is a near-balanced edge set with $|I_2 \cup \{ e_i \}| = 2|V(I_2 \cup \{ e_i \})|$, which violates the $(2, 1)$-sparsity condition for near-balanced sets. Since $I_1$ and $I_2$ have distinct cardinalities, ${\cal I}_{\alpha}(G)$  does not form the family of independent sets of a matroid.
This example indicates the necessity of the assumption on the value of $\alpha(\mathbb{Z}_2)$ in Theorem~\ref{thm:main}.

Theorem~\ref{thm:main} implies that, even if $\Gamma$ contains an element of order two, the following condition induces a matroid:
\[
|F|\leq 2|V(F)|-3+\begin{cases}
0 & \text{if $F$ is balanced} \\
2 & \text{if $F$ is near-balanced, or $\langle F\rangle_{v, \psi}\simeq \mathbb{Z}_2$ for some $v\in V(F)$} \\
3 & \text{otherwise}.
\end{cases} %\qquad (F\in {\cal C}(G))
\]
Interestingly these additional conditions turn out to be necessary for the infinitesimal rigidity of symmetric bar-joint frameworks~\cite{st, ikeshita}.

\medskip
\noindent {\it Example 4.}
The following count condition appears when analyzing the infinitesimal rigidity of frameworks with dihedral symmetry on the plane~\cite{jkt}:
%In \cite{jet} the following count condition has been extensively analyzed:
\begin{equation*}
|F|\leq 2|V(F)|-3+\begin{cases}
0 & \text{ if $F$ is balanced} \\
2 & \text{ if $\langle F\rangle_{v,\psi}$ is nontrivial and cyclic for some $v\in V(F)$} \\
3 & \text{ otherwise }
  \end{cases} \qquad (F\in {\cal C}(G)).
%   \qquad (F\subseteq E)
\end{equation*}
In \cite{jkt} it was shown that the count induces a matroid when $\Gamma$ is dihedral.
The following lemma gives a condition for the corresponding $\alpha$ to be polymatroidal.
\begin{lemma}
The function $\alpha : 2^\Gamma \to \mathbb{Z}$ defined by 
\begin{equation*}
\alpha(X)=\begin{cases}
0 & \text{ $\langle X\rangle$ is trivial} \\
2 & \text{  $\langle X\rangle$ is nontrivial and cyclic} \\
3 & \text{ otherwise}
\end{cases} \qquad (X\subseteq \Gamma).
\end{equation*}
is  polymatroidal if and only if for each element $g\in \Gamma\setminus \{1_{\Gamma}\}$ 
a maximal cyclic subgroup containing $g$ is unique.
\end{lemma}
\begin{proof}
Note that $\alpha$ satisfies the monotonicity, the invariance under conjugation, and the invariance under taking the closure. We prove that $\alpha$ is submodular if and only if for each element $g\in \Gamma\setminus \{1_{\Gamma}\}$ 
a maximal cyclic subgroup containing $g$ is unique.

Suppose a maximal cyclic group containing each element is unique.
The submodularity can be checked as follows.
Take any $X, Y\subseteq \Gamma$.
If $\langle X\rangle$ or $\langle Y\rangle$ is not cyclic,
the submodular inequality is trivial.
If $\langle X\rangle$ and $\langle Y\rangle$ are nontrivial and cyclic,
there are unique maximal cyclic subgroups $\Gamma_X$ and $\Gamma_Y$ containing $X$ and $Y$, respectively. 
If $\Gamma_X\cap \Gamma_Y=\{1_{\Gamma}\}$, then 
$\alpha(X)+\alpha(Y)=4>3\geq \alpha(X\cap Y)+\alpha(X\cup Y)$.
If $\Gamma_X\cap \Gamma_Y\neq \{1_{\Gamma}\}$, then it is cyclic 
and there is a unique maximal cyclic subgroup containing $\Gamma_X\cap \Gamma_Y$.
However, since $\Gamma_X$ and $\Gamma_Y$ are maximal, 
we have $\Gamma_X=\Gamma_Y$, implying 
$\alpha(X)+\alpha(Y)=\alpha(\Gamma_X)+\alpha(\Gamma_Y)
\geq \alpha(X\cap Y)+\alpha(X\cup Y)$.

Conversely, if there is an element $g\in \Gamma$ that is  contained in two distinct maximal 
cyclic subgroups $\Gamma_1$ and $\Gamma_2$.
Then $\alpha(\Gamma_1\cap \Gamma_2)\geq \alpha(\{g\})\geq 2$
and $\alpha(\Gamma_1\cup \Gamma_2)=3$.
Hence the submodularity does not hold.	
\end{proof}

A dihedral group is an example satisfying this property while 
$\mathbb{Z}_3\times \mathbb{Z}_2\times \mathbb{Z}_3$ is an example 
not having the property.

It was shown in \cite{jkt} that the so-called symmetry-forced rigidity of 2-dimensional bar-joint frameworks with dihedral symmetry with order $2n$ for some odd $n$ can be characterized in terms of  this count condition (under a certain generic assumption).

\medskip
\noindent{\it Example 5}.
Let $n, i$ be positive integers with $i<n$, and let 
\begin{align*} 
S_0(n,i)&=\{n' \in \mathbb{Z}: 2\leq n' \leq n, n' \text{ divides } n \text{ and } i\} \\
S_{-1}(n,i)&=\{n' \in \mathbb{Z}: 2\leq n' \leq n, n' \text{ divides }n \text{ and } i-1\} \\
S_{1}(n,i)&=\{n' \in \mathbb{Z}: 2\leq n' \leq n, n' \text{ divides }n \text{ and } i+1\} 
\end{align*}
\begin{equation*}
S(n,i)=\begin{cases}
S_0(n,i) \cup S_{-1}(n,i) \cup S_1(n,i) & \text{ if $i$ is even}\\
S_0(n,i) \cup S_{-1}(n,i) \cup S_1(n,i)\setminus \{2\} & \text{ if $i$ is odd}.
\end{cases}
\end{equation*}

Suppose that we have a $\mathbb{Z}_n$-labeled graph $(G,\psi)$.
The following count condition appears when analyzing the infinitesimal rigidity of frameworks with cyclic symmetry:
\begin{equation*}
|F|\leq 2|V(F)|-3+\begin{cases}
0 & \text{ if $F$ is balanced} \\
1 & \text{ if  $i$ is odd and $\langle F\rangle_{v,\psi}\simeq \mathbb{Z}_2$  for some $v\in V(F)$} \\
2 & \text{ if $\langle F\rangle_{v,\psi}\simeq \mathbb{Z}_k$ for some $k\in S(n,i)$, or
$F$ is near-balanced} \\
3 & \text{ otherwise}.
\end{cases}
\end{equation*}
This count indeed determines a matroid since the corresponding $\alpha$ is polymatroidal as shown below.
\begin{lemma}
The function $\alpha : 2^{\mathbb{Z}_n} \to \mathbb{Z}$ defined by
\begin{equation*}
\alpha(X)=\begin{cases}
0 & \text{ if $\langle X\rangle$ is trivial} \\
1 & \text{ if  $i$ is odd and $\langle X\rangle\simeq \mathbb{Z}_2$} \\
2 & \text{ if $\langle X\rangle\simeq \mathbb{Z}_k$ for some $k\in S(n,i)$} \\
3 & \text{ otherwise }
\end{cases} \qquad (X\subseteq \mathbb{Z}_n).
\end{equation*}
is polymatroidal.
\end{lemma}
\begin{proof}
Only the submodularity of $\alpha$ is nontrivial.
Take any $X, Y \subseteq \Gamma$.
Since $\alpha (\langle X \rangle \cap \langle Y \rangle) + \alpha (\langle X \rangle \cup \langle Y \rangle)
 \geq \alpha (X \cap Y) + \alpha (X \cup Y)$,
it suffices to consider the case when $X$ and $Y$ are subgroups of $\mathbb{Z}_n$.
Let $n_X$ and $n_Y$ be positive integers dividing $n$ such that $X \simeq \mathbb{Z}_{n_X}$ and $Y \simeq \mathbb{Z}_{n_Y}$,
and let $g = \mathrm{gcd} (n_X, n_Y)$ and $l = \mathrm{lcm} (n_X, n_Y)$.
Then we have $X \cap Y = \{ 0, \frac{n}{g}, \ldots, \frac{(g - 1)n}{g} \} \simeq \mathbb{Z}_g$ and
$\langle X \cup Y \rangle  
%\{ a \cdot \frac{n}{n_X} + b \cdot \frac{n}{n_Y} : a, b \in \mathbb{Z} \}  \slash n \mathbb{Z} 
= \mathrm{gcd} (\frac{n}{n_X}, \frac{n}{n_Y}) \mathbb{Z} \slash n \mathbb{Z} = \frac{n}{l} \mathbb{Z} \slash n \mathbb{Z} \simeq \mathbb{Z}_l$,
implying $\alpha (X \cap Y) + \alpha (X \cup Y) \leq \alpha (X \cap Y) + \alpha (\langle X \cup Y \rangle) = \alpha (\mathbb{Z}_g) + \alpha (\mathbb{Z}_l)$.
Hence we need only to show that 
\begin{equation}
\label{eq:cyclic_sub}
\alpha (\mathbb{Z}_{n_X}) + \alpha (\mathbb{Z}_{n_Y}) \geq \alpha (\mathbb{Z}_g) + \alpha (\mathbb{Z}_l).
\end{equation}

Suppose that $i$ is odd.
If $n_X=1$, then $g=1$ and $l=n_Y$, implying (\ref{eq:cyclic_sub}).
Also, if $n_X\notin S(n,i)\cup \{1,2\}$, then $l\notin S(n,i)\cup \{1,2\}$ and hence 
$\alpha (\mathbb{Z}_{n_X})=\alpha (\mathbb{Z}_{l})=3$.
Since $\alpha (\mathbb{Z}_{n_Y})\geq \alpha (\mathbb{Z}_{g})$ always holds, 
we get (\ref{eq:cyclic_sub}).
Therefore, we may suppose that $n_X, n_Y\in S(n,i)\cup \{2\}$.

If $n_X=n_Y=2$, then $g=l=2$, and hence (\ref{eq:cyclic_sub}) follows.

If $n_X\in S(n,i)$ and $n_Y=2$, then $g\leq 2$. 
When $g=1$, $\alpha (\mathbb{Z}_{n_X}) + \alpha (\mathbb{Z}_{n_Y})=3\geq \alpha (\mathbb{Z}_l) =\alpha (\mathbb{Z}_g) + \alpha (\mathbb{Z}_l)$.
When $g=2$,  $l=n_X$ and $g=n_Y$ hold, and thus (\ref{eq:cyclic_sub}) holds.

Suppose finally that $n_X\in S(n,i)$ and $n_Y\in S(n,i)$.
If $g\notin S(n,i)$, then $\alpha (\mathbb{Z}_{n_X}) + \alpha (\mathbb{Z}_{n_Y})-\alpha (\mathbb{Z}_g)\geq 3\geq  \alpha (\mathbb{Z}_l)$.
On the other hand, if $g\in S(n,i)$, then $l\in S(n,i)$ holds, which implies (\ref{eq:cyclic_sub}). 
Indeed, if $n_X\in S_{j_X}(n,i)$ and $n_Y\in S_{j_Y}(n,i)$ for some $j_X, j_Y\in \{-1,0,1\}$,
then $j_X-j_Y$ is an integer multiple of $g$.
Since $g>2$ by $g\in S(n,i)$, this implies $j_X=j_Y$, and hence $l\in S(n,i)$ holds as we claimed.

Suppose that $i$ is even.
We can do the same case analysis as in the case of odd $i$, and the only nontrivial case is when $n_X, n_Y, g\in S(n,i)$.
We again show $l\in S(n,i)$.
Let $j_X$ and $j_Y$ be as above. 
Then $j_X-j_Y$ is an integer multiple of $g$.
Since $j_X=j_Y$ implies $l\in S(n,i)$, assume $j_X\neq j_Y$.
Since $g>1$, we have $g=2$ and $j_Xj_Y=-1$.
However, since $i$ is even, $i+j_X$ and $i+j_Y$ are both odd.
Since $n_X$ and $n_Y$ divid $i+j_X$ and $i+j_Y$, respectively, 
$g$ must be odd, contradicting $g=2$.
Therefore, $j_X=j_Y$ always holds, and $l\in S(n,i)$ implies (\ref{eq:cyclic_sub}).
\end{proof}

It was shown in \cite{ikeshita} that the infinitesimal rigidity of 2-dimensional bar-joint frameworks with cyclic symmetry of odd order $n$ can be characterized in terms of  these count conditions (under a certain generic assumption).

\section{Near-balancedness}
In this section we shall prepare notation and present several properties of near-balancedness.

Let $G=(V,E)$ be a connected graph. For $F\subseteq E(G)$ and  $v\in V(F)$ let $F_v$ be the set of edges in $F$ incident to $v$, and  let $G_F=(V(F), F)$. 
For $v\in V$, we denote by $L_v$  the set of loops in $G$ incident to $v$, 
and by $L_v^\circ$ the set of balanced loops incident to $v$.
For a vertex $v$, the subgraph of $G-L_v$ induced by $v$ and the vertex set of a connected component of $G-v$ is called a {\em fraction} of $v$. Note that if $v$ is not a cut vertex then $G-L_v$ is a fraction of $v$.
%For a vertex  $v$,  a {\em fraction} of $v$ is a subgraph of $G$ which is either a single-vertex graph with one loop at $v$ or the subgraph of $G-L_v$ induced by $v$ and the vertex set of a connected component of $G-v$. 
%A fraction consisting of a single vertex is said to be {\em trivial}.

Let $(G,\psi)$ be a $\Gamma$-labeled graph.
For $v\in V(G)$ and $g\in \Gamma$, a {\em switching} at $v$ with $g$ is an operation that creates a new gain function $\psi'$ from  $\psi$ as follows:
\begin{equation*}
\psi'(e)=\begin{cases}
g\cdot \psi(e)\cdot g^{-1} & \text{if $e$ is a loop incident with $v$} \\
g\cdot \psi(e) & \text{if $e$ is a non-loop edge and is directed from $v$} \\
\psi(e)\cdot g^{-1} & \text{if $e$ is a non-loop edge and is directed to $v$}\\
\psi(e) & \text{otherwise}
\end{cases} \qquad (e\in E(G)).
\end{equation*}
A gain function $\psi$ is said to be {\em equivalent} to $\psi'$ if $\psi$ can be obtained from $\psi'$ by a sequence of switchings.
It is easy to see that $\langle F\rangle_{v, \psi}$ is conjugate to $\langle F\rangle_{v, \psi'}$  for any equivalent $\psi$ and $\psi'$. (See, e.g.,~\cite[Section 2.5.2]{gt}.)

For a forest $F\subseteq E(G)$, a gain function $\psi'$ is said to be {\em $F$-respecting} if $\psi'(e)=1_{\Gamma}$ for every $e\in F$.
For any forest $F\subseteq E(G)$, there always exists an $F$-respecting gain function equivalent to $\psi$. 
A frequently used fact in the subsequence discussion is that, if $\psi'$ is $T$-respecting for a spanning tree $T$ of $G_F$, then  $\langle F\rangle_{v,\psi'}=\langle \psi'(F)\rangle$ for any $v\in V(F)$, where $\psi'(F)=\{\psi'(e): e\in F\}$ (see, e.g.,~\cite[Section 2.2]{jkt}). Hence $\tilde{\alpha}(F)=\alpha(\psi'(F))$.
%For any $\gamma\in \Gamma$, switching operations at all vertices with $\gamma$ results in an equivalent gain function $\psi'$ with $\psi'(e)=\gamma \psi(e)\gamma^{-1}$.
%Such a gain function is said to be {\em conjugate} of $\psi$.

%We say that $(G,\psi)$ is $(\alpha, \Gamma')$-critical (or near-$[\Gamma']$-balanced) if so is $E(G)$.
%For a subgroup $\Gamma'$ of $\Gamma$, let $N(\Gamma')$ be the normalizer of $\Gamma$,
%i.e., $N(\Gamma')=\{g\in \Gamma: \Gamma' g=g\Gamma'\}$.
We say that a $\Gamma$-labeled graph $(G,\psi)$ is {\em near-balanced} if $E(G)$ is near-balanced.
The following proposition gives an alternative definition for near-balancedness.

\begin{proposition}
\label{prop:0}
Let $(G,\psi)$ be a connected and unbalanced $\Gamma$-labeled graph with $G=(V,E)$.
Then $(G,\psi)$ is near-balanced if and only if 
there are $v \in V$, $g \in \Gamma\setminus \{1_{\Gamma}\}$, $E_v'\subseteq E_v$,   and an equivalent gain function $\psi'$ such that, assuming that all edges incident to $v$ are directed to $v$,
\begin{itemize}
\item $\psi'(e)=1_{\Gamma}$ for $e\in E\setminus E_v'$, and
\item $\psi'(e)=g$ for $e\in E_v'$.
%\item $[E]=[\Gamma'\cup \{g\}]$.
\end{itemize}
%Conversely, if there is an equivalent gain function $\psi'$ satisfying the first and the second conditions, then $(G,\psi')$ is near-$[\Gamma'']$-balanced for some subgroup $\Gamma''$ of $\Gamma'$.
\end{proposition}
\begin{proof}
Suppose that the split $(H,\psi)$ of $(G,\psi)$ at $v\in V$ with a partition $\{E_1, E_2\}$ of $E_v\setminus L_v$ results in a balanced graph. Let $v_1$ and $v_2$ be the new vertices after the split. If $H$ is disconnected, then $G$ can be obtained from $H$ by identifying $v_1$ and $v_2$,
and hence $(G,\psi)$ turns out to be balanced, which is a contradiction.
Hence $H$ is connected.

Take a spanning tree $T$ of $G$ such that $T\setminus E_2$ is a maximal forest of $G-E_2$, and consider a $T$-respecting equivalent gain function $\psi'$.
Note that $(H,\psi')$ is still balanced.
Let ${\cal G}_1$ be the family of  fractions $G'$ of $v$ in $(G, \psi')$ with $E_1\cap E(G')\neq \emptyset$, and let $E_2'=\{e\in E_2\cap E(G') : G' \in {\cal G}_1\}$.  
We show that $\psi'$ satisfies the property of the statement for 
$E_v':=E_2'\cup(L_v\setminus L_v^{\circ})$.

The first condition of the statement can be checked as follows.
Since $T$ spans $V(H)-v_2$ in $H$ and $(H,\psi')$ is balanced, $\psi'(e)=1_{\Gamma}$ holds for every $e\in E\setminus (E_2\cup (L_v\setminus L_v^{\circ}))$.
Also, for every $e\in E_2\setminus E_2'$, 
the fraction $G'$ of $v$ in $(G,\psi')$ containing $e$ satisfies  $E_1\cap E(G')=\emptyset$ by $e\notin E_2'$.
Hence $(E_2\cap E(G'))\cap T\neq \emptyset$ should hold as $T$ is spanning. 
Since $\psi'$ is $T$-respecting and $(H,\psi')$ is balanced,  
we have $\psi'(e)=1_{\Gamma}$ for $e\in E_2\setminus E_2'$.
Thus $\psi'(e)=1_{\Gamma}$ holds for every $e\in E\setminus E_v'$.

To see the second condition, we pick any $e\in E_v'$ and let $g=\psi'(e)$.
Now, observe that for each $f\in E_v'\setminus \{e\} (=(E_2'\cup (L_v\setminus L_v^\circ))\setminus \{e\})$,
$H$ contains a closed walk starting at $v_2$ and consisting of $e, f$ and  edges in $T$.
See Figure~\ref{fig:prop0}.
This implies $\psi'(e)^{-1}\psi'(f)=1_{\Gamma}$, meaning $\psi'(f)=\psi'(e)=g$.
Thus $\psi'$ is a required equivalent gain function.

\begin{figure}
\centering
\begin{minipage}{0.4\textwidth}
\centering
\includegraphics[scale=0.8]{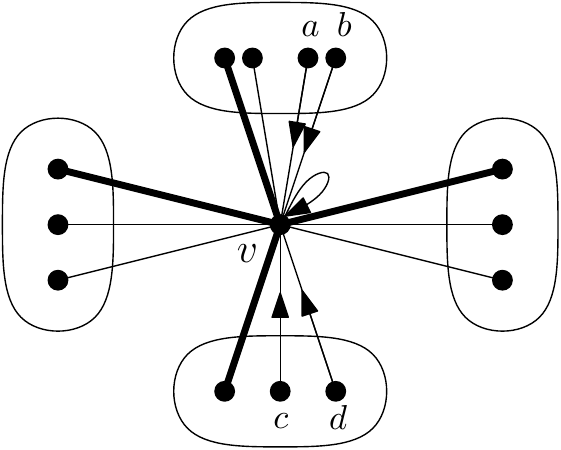}
\par (a)
\end{minipage}
\begin{minipage}{0.4\textwidth}
\centering
\includegraphics[scale=0.8]{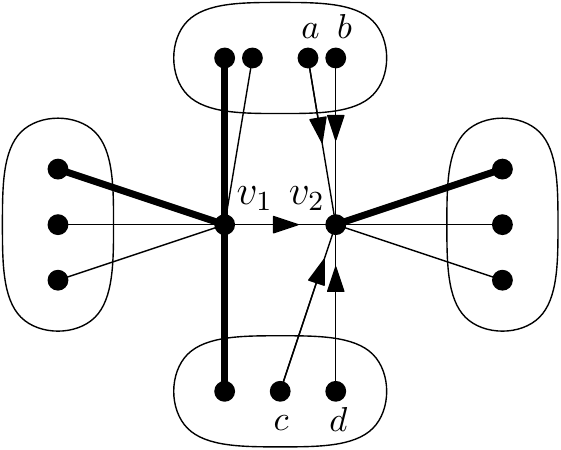}
\par(b) 
\end{minipage}
\caption{The proof of Proposition~\ref{prop:0}. (a) $(G,\psi')$ and (b) its split $(H, \psi')$ at $v$. 
Every unoriented edge has the identity label and $E_2'=\{va,vb,vc,vd\}$. 
The bold edges represent edges in $T$. 
Note that the fraction of $v$ on the right side of $v$ does not belong to ${\cal G}_1$. }
\label{fig:prop0}
\end{figure}
Conversely, if  there are $v\in V$, $g\in \Gamma\setminus \{1_{\Gamma}\}$, $E_v'\subseteq E_v$, and an equivalent gain function $\psi'$ satisfying the statement,
then we let  $E_1=E_v\setminus (E_v'\cup L_v)$ and $E_2=E_v'\setminus L_v$.
We consider the split of  $(G,\psi')$ at $v$ with the partition $\{E_1, E_2\}$ of $E_v\setminus L_v$.
Then the resulting graph is balanced.
\end{proof}

Suppose that $(G,\psi)$ is  near-balanced.
Then there is a balanced split of $(G,\psi)$ at $v\in V(G)$ with a partition $\{E_1, E_2\}$ of $E_v\setminus L_v$.
This $v$ is called a {\em base} for the  near-balancedness and 
$E_2\cup (L_v\setminus L_v^\circ)$  (or $E_1\cup (L_v\setminus L_v^{\circ}$)) is called an {\em extra edge set}.
%Note that, if $F\subseteq E_v$ is an extra edge set, then 
%$(E_v\setminus L_v^{\circ})\setminus (F\setminus L_v)$ is an extra edge set.
%Note that $E_1$ is also an extra edge set, which is said to be the {\em complement} of $E_2$.

The proof of Proposition~\ref{prop:0} also implies the following useful fact.
\begin{proposition}
\label{prop:1}
Let $(G,\psi)$ be a connected near-balanced graph and 
let $E'$ be an extra edge set for the near-balancedness.
Suppose that $\psi$ is $T$-respecting for some spanning tree $T\subseteq E$ with $T\cap E'=\emptyset$.
Then $\psi$ satisfies the following.
\begin{itemize}
\item There is a  nonidentity element $g\in \Gamma$ such that $\psi(e)=g$ for every $e\in E'$.
\item $\psi(e)=1_{\Gamma}$ for $e\in E\setminus E'$. 
\end{itemize}
\end{proposition}

% there are $v\in V(F)$, $\gamma\in \Gamma$, and equivalent $\psi'$ such that,
%assuming that all edges in $F_v$ are directed to $v$,  
%each edge in $F'=\{e\in F \mid \psi'(e)\notin \Gamma'\}$ is incident to $v$ with the gain contained in $\Gamma'g$ and $[F\setminus F']=[\Gamma']$.
%This $v$ is called a {\em base} for the  near-$[\Gamma']$-balancedness and $F'$ is called an {\em extra edge set}.  Note that, if $F'$ is an extra edge set, then $F_v\setminus F'$ is also an extra edge  set for the near-$[\Gamma']$-balancedness. %This can be checked by switching at $v$ with gain $g$.

\section{Main Theorem} \label{sec:proof}
Let $k$ and $\ell$ be positive integers with $\ell\leq 2k-1$.
Our main theorem given below is described under the following smoothness condition on a normalized polymatroidal function $\alpha:\Gamma\rightarrow \{0,1,\dots, \ell\}$:
%for any $g\in \Gamma\setminus \{1_{\Gamma}\}$, 
%\begin{equation}
%\label{eq:assumption0}
%\alpha(\{g\})>k \quad \Rightarrow \quad g^2\neq 1_{\Gamma}.
%\end{equation}
%By the submodularity, this is equivalent to saying that, 
 for any $\emptyset\neq S\subseteq \Gamma$ and $g\in \Gamma$, 
\begin{equation}
\label{eq:assumption}
\alpha(S\cup \{g\})-\alpha(S)>k \quad \Rightarrow \quad S=\{1_\Gamma\} \text{ and } g^2\neq 1_{\Gamma}.
\end{equation}
Since $\alpha$ is normalized,  we have $\alpha(\{g\})>0$ for any non-identity $g\in \Gamma$. Hence, if $\ell\leq k+1$, then (\ref{eq:assumption}) is equivalent to
\begin{equation}
\label{eq:assumption2}
\alpha(\Gamma')\leq k \quad \text{for any subgroup $\Gamma'\subseteq \Gamma$ isomorphic to $\mathbb{Z}_2$.} 
\end{equation}
Now we are ready to state our main theorem.
\begin{theorem}
\label{thm:main1}
Let $k, \ell$ be integers with $k\geq 1$ and $0\leq \ell\leq 2k-1$, $(G,\psi)$ be a $\Gamma$-labeled graph, and $\alpha:2^{\Gamma}\rightarrow \{0,1,\dots, \ell\}$ be a normalized polymatroidal function satisfying the smoothness condition (\ref{eq:assumption}),
and define $f_{\alpha}:{\cal C}(G)\rightarrow \mathbb{Z}$ by
\[
f_{\alpha}(F)=k|V(F)|-\ell+
\begin{cases}
\min\{\tilde{\alpha}(F), k\} & (\text{if $F$ is near-balanced}) \\
\tilde{\alpha}(F) & (\text{otherwise}).
\end{cases}
\]
Then the set ${\cal I}_{\alpha}(G)=\{I\subseteq E(G)\mid |F|\leq f_{\alpha}(F)\ \forall F\in {\cal C}(G)\cap 2^I\}$
forms the family of independent sets in a matroid.
\end{theorem}
The case when  $\ell\leq k+1$ implies Theorem~\ref{thm:main}  due to the equivalence between (\ref{eq:assumption}) and (\ref{eq:assumption2}).

Before moving to the proof, we give a remark on the technical difference between Theorem~\ref{thm:main1} and the previous work. 
In \cite{t} the second author proved Theorem~\ref{thm:tan} (corresponding to the case for $\ell=k$) by showing that a set function $\hat{f}_{\alpha}:2^E\rightarrow \mathbb{R}$ defined by 
\[
\hat{f}_{\alpha}(F)=\sum_{C: \text{ connected component of } F} f_{\alpha}(C) \qquad (F\subseteq E)
\]
is monotone submodular. 
Then the theorem immediately  follows from Edmonds' theorem~\cite{e70} on intersecting submodular functions. However, for $\ell>k$, $\hat{f}_{\alpha}$ may not be submodular in general and we do not know whether our main theorem (Theorem~\ref{thm:main1}) is a consequence of a general theory of intersecting submodular functions.
In \cite{jkt} a special case (given in Example 4) was proved by directly checking the independence axiom, and here we will follow the same approach. 

The main observation in the proof is Lemma~\ref{lem:4}, which asserts the submodular relation among sets that intersect "nicely". To prove this, we further investigate properties of near-balanced graphs in Subsection~\ref{subsec:further}, and then we move to a proof of Theorem~\ref{thm:main1} in Subsection~\ref{subsec:proof}.
%By Theorem~\ref{thm:tan} we assume $\ell<k$ in the subsequent discussion.

For simplicity of description, denote $\beta:{\cal C}(G)\rightarrow \mathbb{Z}$ by
\[
\beta(F)=\begin{cases}
\min\{\tilde{\alpha}(F), k\} & (\text{if $F$ is near-balaced}) \\ 
\tilde{\alpha}(F) & (\text{otherwise})
\end{cases} \qquad (F\in {\cal C}(G)).
\]
We say that $(G,\psi)$ is {\em $f_{\alpha}$-sparse} if 
$|F|\leq f_{\alpha}(F)$ holds for every $F\in {\cal C}(G)$.
A $\Gamma$-labeled graph $(G,\psi)$ is called  {\em $f_{\alpha}$-tight} 
if it is connected $f_{\alpha}$-sparse with $|E(G)|=f_{\alpha}(E(G))$.
Also  $(G,\psi)$ is called {\em $f_{\alpha}$-full} if it contains a connected $f_{\alpha}$-sparse subgraph $G'$ such that 
\begin{itemize}
\item $G'$ is spanning, i.e., $V(G')=V(G)$,
\item  $\beta(E(G'))=\beta(E(G))$, and
\item $|E(G')|\geq k|V(G')|-\ell+\min\{\beta(E(G')), 2k-\ell+1\}$.
\end{itemize}
Note that any $f_{\alpha}$-tight graph is $f_{\alpha}$-full.
An edge set $F$ is called $f_{\alpha}$-sparse, $f_{\alpha}$-tight, and $f_{\alpha}$-full, respectively, if 
so is the induced subgraph $G_F$.

\subsection{Further properties of near-balancedness}
\label{subsec:further}
Assuming $f_{\alpha}$-fullness, near-balanced graphs have further nice properties.
In the subsequent discussion, $\alpha$ always denotes a normalized polymatroidal function.
%\begin{lemma}
%\label{lem:full}
%Suppose that $(G,\psi)$ is $f_{\alpha}$-full with $\beta(E(G))\geq 2k-\ell+1$. 
%Then it contains a spanning $f_{\alpha}$-tight subgraph $(G',\psi)$
%with $|E(G')|>k|V(G)|-2\ell+2k$ and $\beta(E(G))=\beta(E(G'))$.
%\end{lemma}
%\begin{proof}
%By definition, $(G,\psi)$ contains a spanning $f_{\alpha}$-tight subgraph $(G',\psi)$
%with $\beta(E(G))=\beta(E(G'))$.
%By assumption on $\beta(E(G))$,  we have $|E(G')|=k|V(G)|-\ell+\beta(E(X'))
%=k|V(G)|-\ell+\beta(E(X))
%>k|V(G)|-2\ell+2k$.
%\end{proof}

%\begin{lemma}
%\label{lem:tight0}
%Suppose that $(G,\psi)$ is near-balanced and $f_{\alpha}$-full.
%Then any connected subgraph $G'$ of $G$ satisfying $\beta(E(G'))=\beta(E(G)))$ is near-balanced.
%\end{lemma}
%\begin{proof}
%By definition of near-balancedness, any connected subgraph of $G$ is either near-balanced or balanced.
%Since $G$ is not balanced, $[E(G)]$ is nontrivial. 
%Since $\alpha$ is normalized, (c6) implies $\tilde{\alpha}(E(G))>0$, which further implies 
%$\beta(E(G'))=\beta(E(G))= \min\{k,\tilde{\alpha}(E(G))\}>0$. 
%Thus $G'$ is not balanced, and  hence is near-balanced. 
%\end{proof}

\begin{lemma}
\label{lem:tight1}
Suppose that $(G,\psi)$ is near-balanced and $f_{\alpha}$-full with $\beta(E(G))\geq 2k-\ell+1$. 
Then a base for the near-balancedness is unique.
\end{lemma}
\begin{proof}
By definition, $(G,\psi)$ 
contains a spanning connected $f_{\alpha}$-sparse subgraph $(G',\psi)$ with 
\begin{equation}
\label{eq:tight10}
|E(G')|\geq k|V(G)|-2\ell+2k+1
\end{equation}
and  $\beta(E(G'))=\beta(E(G))$. 
Note that 
 $(G',\psi)$ is also near-balanced, since otherwise $(G',\psi)$ would be balanced 
 and $0=\beta(E(G'))=\beta(E(G))=\tilde{\alpha}(E(G))$, contradicting that $(G,\psi)$ is unbalanced.
Thus it suffices to show the uniqueness of the base for $(G',\psi)$.
Let $E'=E(G')$.

Suppose that there are two distinct base vertices $u$ and $v$ for the near-balancedness of $(G',\psi)$. 
Clearly $G'$ cannot contain an unbalanced loop since otherwise, say if $u$ is incident to an unbalanced loop,  then any split at $v$ cannot be balanced.
Without loss of generality, assume that all edges incident to $v$ are directed to $v$.
By Proposition~\ref{prop:0} there are $g\in \Gamma\setminus \{1_{\Gamma}\}$, $F_v\subseteq E'_v$, and an equivalent gain function $\psi'$ such that 
\begin{equation}
\label{eq:tight11}
\text{$\psi'(e)=g$ for $e\in F_v$ and  $\psi'(e)=1_{\Gamma}$ for $e\in E\setminus F_v$}.
\end{equation}
%Since every loop has the identity label, $F_v\cap L_v=\emptyset$.
%Let $E'$ be the graph obtained by taking the union of all simple closed walks $W$ (regarding them as subgraphs) with the properties that $\psi(W)=1_{\Gamma}$ and $W$ does not pass through $u$ as an intermediate node.
Note also that 
\begin{equation}
\label{eq:tight12}
\text{$F_v\neq \emptyset$ and $E_v'\setminus F_v\neq \emptyset$}, 
\end{equation}
since otherwise $G$ would be balanced.

Let $K$ be the union of the edge sets of all simple walks $W$ in $G'$ starting at $v$  with the following property:
\begin{equation}
\label{eq:tight1prop} 
\text{$\psi'(W)=g^{-1}$ and $W$ does not contain  $u$ as an internal node (but may be the last).}
\end{equation}
By (\ref{eq:tight11}), $F_v\subseteq K$ and $E_v'\setminus F_v \subseteq E'\setminus K$.
Hence, (\ref{eq:tight11}) again implies that $K$ and $E'\setminus K$ are balanced.
Since they are also nonempty by (\ref{eq:tight12}), 
we get 
\begin{equation}
\label{eq:tight13}
\text{$|K|\leq k|V(K)|-\ell$ and $|E'\setminus K|\leq k|V(E'\setminus K)|-\ell$}
\end{equation} 
by $f_{\alpha}$-sparsity.
We also claim that 
\begin{equation}
\label{eq:tight14}
V(K)\cap V(E'\setminus K)\subseteq \{u,v\}.
\end{equation}
To see this,  suppose  that there is a vertex $w\in V(K)\cap V(E'\setminus K)$ other than $u$ and $v$,
and let $e'$ be an edge of $E'\setminus K$ incident to $w$.
Then the other endvertex of $e'$ should be $v$ since otherwise there would be a simple walk passing $e'$ and satisfying (\ref{eq:tight1prop}).
However, by $w\in V(K)$, the concatenation of $e'$ and a simple path from $v$ to $w$ with gain $g^{-1}$ is an unbalanced cycle which does not pass through $u$,
contradicting that a split of $(G,\psi)$ at $u$ results in a balanced graph.
Hence (\ref{eq:tight14}) holds. 
Combining (\ref{eq:tight13}) and (\ref{eq:tight14}), we get 
$|E'|=|K|+|E'\setminus K|\leq k|V(K)|+k|V(E'\setminus K)|-2\ell
\leq k|V(E')|-2\ell+2k=k|V(G)|-2\ell+2k$, which contradicts (\ref{eq:tight10}).
\end{proof}

\begin{lemma}
\label{lem:near1}
Suppose that $(G,\psi)$ is near-balanced and $f_{\alpha}$-full with $\beta(E(G))\geq 2k-\ell+1$.
Then each fraction of a base $v$ is near-balanced.
In particular, for each extra edge set $K$ of the near-balancedness, $G-K$ is connected.
\end{lemma}
\begin{proof}
It suffices to show that each fraction $S$ of $v$ is unbalanced.
Suppose that $S$ is balanced.
By definition, $(G,\psi)$ 
contains a  spanning $f_{\alpha}$-sparse subgraph $(G',\psi)$ with $|E(G')|\geq k|V(G)|-2\ell+2k+1$. 
Then 
\begin{align*}
|E(G')|&=|E(G')\setminus E(S)|+|E(G')\cap E(S)| && \\
&\leq k|(V(G')\setminus V(S))\cup \{v\}|-\ell+k
+k|V(G')\cap V(S)|-\ell && \text{(by $f_{\alpha}$-sparsity)} \\
&=k |V(G)|-2\ell+2k,  && \text{(since $S$ is a fraction)},
\end{align*}
which is a contradiction.
%However,  the lemma assumption gives $|E(G')|>k|V(G')-2\ell+2k$, a contradiction.
Hence $S$ is unbalanced.
\end{proof}

A $\Gamma$-labeled graph $(G,\psi)$ (resp.~an edge set $E$) is called {\em $\alpha$-critical} if it is connected and near-balanced with $\tilde{\alpha}(E(G))>k$. 
If $(G,\psi)$ is  $\alpha$-critical, then 
$\ell\geq \tilde{\alpha}(E(G))>k$ and hence $\beta(E(G))=k>2k-\ell$ follows.
This in turn implies that an $\alpha$-critical graph always satisfies the assumption for $\beta(E(G))$ in  Lemmas~\ref{lem:tight1} and~\ref{lem:near1}.

The following lemma (Lemma~\ref{lem:tight2}) says that, for an $\alpha$-critical graph, even an extra edge set for the near-balancedness is uniquely determined (up to complementation of non-loop edges).

\begin{lemma}
\label{lem:tight2}
Suppose that $(G,\psi)$ is $\alpha$-critical and $f_{\alpha}$-full, and let  $v$ be the base.
If there are two distinct extra edge sets $E_1$ and $E_2$ for the  near-balancedness, 
then  $\{E_1\setminus L_v, E_2\setminus L_v\}$ is a partition of $E_v\setminus L_v$.
\end{lemma}
 \begin{proof}
 By Lemma~\ref{lem:near1},  $G-E_1$ is connected and hence $G$ contains a spanning tree $T$ with $T\cap E_1=\emptyset$. 
 We may assume that $\psi$ is $T$-respecting.
Then there is an element  $g\in \Gamma\setminus \{1_{\Gamma}\}$ such that $\psi(e)=g$ for $e\in E_1$ and $\psi(e)=1_{\Gamma}$ for $e\in E\setminus E_1$ by Proposition~\ref{prop:1}. 

 %Note that $G'$ is $(\alpha,[\Gamma_i'])$-critical for some subgroup $\Gamma_i'$ of $\Gamma_i$.
 Let $S$ be a fraction of $v$.
 By Lemma~\ref{lem:near1}, $\emptyset \neq E_i\cap E(S)\neq E_v\cap E(S)$ for $i=1,2$.
 %Consider the split $(H',\psi)$ of $(G',\psi)$ at $v$ with the partition of $\{E_1', (E_v\cap E(G'))\setminus E_1',\emptyset\}$.  
 Since $S-v$ is connected, if $E_2\cap E(S)$ contains an edge with label $g$ and an edge with label $1_{\Gamma}$, then the split of $(S,\psi)$ at $v$ with the partition of $\{E_2\cap E(S), (E_v\cap E(S))\setminus E_2\}$ contains an unbalanced cycle, which contradicts that the split is balanced.
 Similarly, $(E_v\cap E(S))\setminus E_2$ cannot contain an edge with label $g$ and an edge with label $1_{\Gamma}$ simultaneously. 
 These imply that 
\begin{description}
 \item[(i)] $E_1\cap E(S)=E_2\cap E(S)$, or 
 \item[(ii)] $\{E_1\cap E(S), E_2\cap E(S)\}$ is a partition of $E_{v}\cap E(S)$
 \end{description}
 for each fraction $S$ of $v$.
 
 \begin{figure}[t]
\centering
\begin{minipage}{0.32\textwidth}
\centering
\includegraphics[scale=0.8]{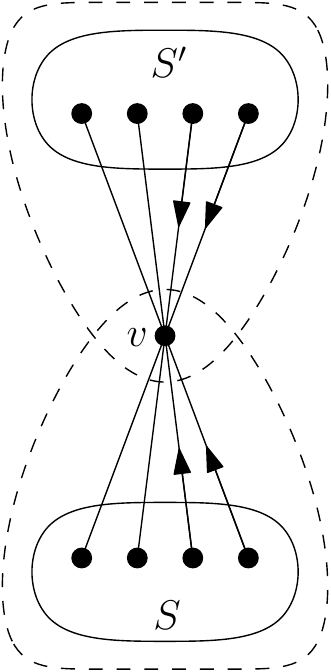}
\par
(a)
\end{minipage}
\begin{minipage}{0.32\textwidth}
\centering
\includegraphics[scale=0.8]{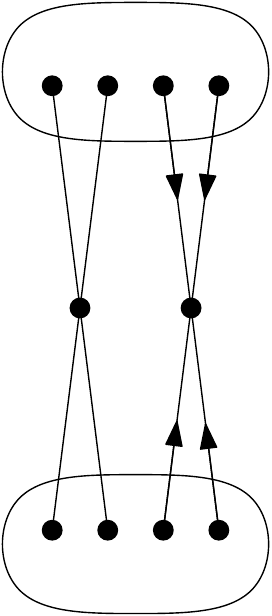}
\par
(b)
\end{minipage}
\begin{minipage}{0.32\textwidth}
\centering
\includegraphics[scale=0.8]{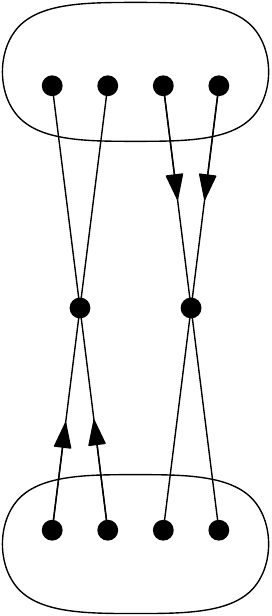}
\par
(c)
\end{minipage}
\caption{Proof of Lemma~\ref{lem:tight2}. 
(a) $(S\cup S', \psi)$, (b) the splitting at $v$ with the partition $\{E_1\cap E(S\cup S'), (E_v\cap E(S\cup S'))\setminus E_1\}$, and (c) the splitting at $v$ with the partition $\{E_2\cap E(S\cup S'), (E_v\cap E(S\cup S'))\setminus E_2\}$, where the oriented edges have the label $g$ and other edges have the identity label.  }
\label{fig:lem44}
\end{figure}
 Since $E_1\neq E_2$, there is a fraction $S$ of $v$ satisfying (ii).
 If there is another fraction $S'$ of $v$ satisfying (i),
 then the split of $(S\cup S',\psi)$ at $v$ with the partition $\{E_2\cap E(S\cup S'), (E_v\cap E(S\cup S'))\setminus E_2\}$  contains a closed walk with gain $g^2$. 
 See Figure~\ref{fig:lem44}.
 Thus $g^2=1_{\Gamma}$.
 However, since $G$ is $\alpha$-critical, $\tilde{\alpha}(E)=\alpha(\{g\})>k$.
 This contradicts the smoothness assumption (\ref{eq:assumption}).
 Therefore each fraction satisfies (ii), and  $\{E_1\setminus L_v, E_2\setminus L_v\}$ is a partition of $E_v\setminus L_v$.
% 
% $L_v\cap E_1=L_v\cap E_2$ should hold since otherwise any split would contain an unbalanced loop.
 \end{proof}

We also remark the following easy lemma.

\begin{lemma}
\label{lem:alpha1}
Suppose that $(G,\psi)$ is $\alpha$-critical.
Then any connected subgraph of $(G,\psi)$ is either $\alpha$-critical or balanced.
\end{lemma}
\begin{proof}
An $\alpha$-critical graph $(G,\psi)$ is near-balanced, and hence by Proposition~\ref{prop:0} there are  $v\in V$,  $g\in \Gamma\setminus \{1_{\Gamma}\}$, $E_v'\subseteq E_v$, and an equivalent gain function $\psi'$ such that, assuming that all edges incident to $v$ are directed to $v$, $\psi'(e)=g$ for $e\in E_v'$ and $\psi'(e)=1_{\Gamma}$ for $e\in E\setminus E_v'$.
Therefore $\alpha(\{g\})=\tilde{\alpha}(E(G))>k$.

If a connected subgraph $G'$ is not balanced, then it contains a walk of gain $g$. Thus $\tilde{\alpha}(E(G'))>k$.
Clearly $G'$ is near-balanced, and hence it is $\alpha$-critical.
\end{proof} 
 
\subsection{Proof of Theorem~\ref{thm:main1}}
\label{subsec:proof}
%\begin{lemma}
%\label{lem:con}
%Suppose that $\ell<k$ and $(G,\psi)$ is balanced with $|E|>f_{\alpha}(E)-\ell+k$.
%Then $G$ is 2-connected.
%\end{lemma}
%\begin{proof}
%If $G$ is not 2-connected, then there are balanced subgraphs $G_1$ and $G_2$ of $G$
%such that $|V(G_1)\cap V(G_2)|\leq 1$ and $\{E(G_1), E(G_2)\}$ partitions $E(G)$.
%We then have $|E|\leq k|V(G_1)|+k|V(G_2)|-2\ell\leq k|V|-2\ell+k=f_{\alpha}(E)-\ell+k$ by $\ell>k$,
%contradicting the lemma assumption.
%\end{proof}

The proof of Theorem~\ref{thm:main1} follows from 
Lemma~\ref{lem:main1} and Lemma~\ref{lem:main2}, which are analogs of well-known properties of $(k,\ell)$-sparse graphs.
The core of the proofs of those two lemmas is the following hidden submodularity of $\beta$.

\begin{lemma}
\label{lem:4}
Suppose that $X, Y\in {\cal C}(G)$ are $f_{\alpha}$-full sets such that
\begin{itemize}
\item $(V(X) \cap V(Y), X\cap Y)$ is connected,
\item $X\cap Y$ is $f_{\alpha}$-sparse, and 
\item $|X\cap Y|>k|V(X\cap Y)|-2\ell+\min\{2k, \beta(X)+\beta(Y)\}$.
\end{itemize}
Then 
$\beta(X)+\beta(Y)\geq \beta(X\cap Y)+\beta(X\cup Y)$.
\end{lemma}
\begin{proof}
Since $(V(X)\cap V(Y), X\cap Y)$ is connected, $G_{X\cap Y}=(V(X)\cap V(Y), X\cap Y)$ holds,
and  there is a spanning tree $T\subseteq X\cup Y$ of $G_{X\cup Y}$ such that $T\cap X, T\cap Y$, and $T\cap X\cap Y$ are spanning trees of $G_X$,
$G_Y$, and $G_{X\cap Y}$, respectively. 
We may assume that $\psi$ is $T$-respecting.
Then we have 
\begin{align}
\nonumber
\tilde{\alpha}(X)+\tilde{\alpha}(Y)&=
\alpha(\psi(X))+\alpha(\psi(Y)) && (\text{by (c5)}) \\ \nonumber
&\geq 
\alpha(\psi(X)\cap \psi(Y))+
\alpha(\psi(X)\cup \psi(Y)) && (\text{by (c2)}) \\ \nonumber
&\geq 
\alpha(\psi(X\cap Y))+\alpha(\psi(X\cup Y)) && ( \text{by (c3)}) \\
&=\tilde{\alpha}(X\cap Y)+\tilde{\alpha}(X\cup Y) && (\text{by (c5)}). \label{eq:sub0}
\end{align}

We split the proof into three cases.

\medskip

\noindent
(Case 1)
Suppose that neither $X$ nor $Y$ are $\alpha$-critical.
Then by (\ref{eq:sub0}) we have
$\beta(X)+\beta(Y)=\tilde{\alpha}(X)+\tilde{\alpha}(Y)\geq 
\tilde{\alpha}(X\cap Y)+\tilde{\alpha}(X\cup Y)=\beta(X\cap Y)+\beta(X\cup Y)$.

\medskip
\noindent
(Case 2) 
Suppose that $X$ is $\alpha$-critical but $Y$ is not $\alpha$-critical.
Let $v$ be the base and $X_v'$ be an extra edge set for the near-balancedness of $X$.
Also let $Z$ be  the set of all non-loop edges of $X\cap Y$ incident to $v$.
Since $((X_v\setminus X_v')\cup L_v)\setminus L_v^{\circ}$ is an extra edge set of $X$, we can always take $X_v'$ such that 
\begin{equation}
\label{eq:lem45_0}
\text{$Z\setminus X_v'\neq \emptyset$ if $Z\neq \emptyset$}.
\end{equation}
We first show
\begin{equation}
\label{eq:lem45_1}
\text{$G_{X\cap Y}-X_v'$ is connected.}
\end{equation}
Suppose not. 
Then $v$ is in $G_{X\cap Y}$ and there is a fraction of $v$ in $G_{X\cap Y}$ which is balanced.
Let $C$ be the edge set of such a fraction. 
Since $v$ is in $G_{X\cap Y}$, $Z$ is nonempty.
Therefore by (\ref{eq:lem45_0}) $\emptyset \neq Z\setminus X_v' \subseteq (X\cap Y)\setminus C$.
Hence, both $C$ and $(X\cap Y)\setminus C$ are nonempty and connected, 
and we get 

\begin{align*}
|X\cap Y|&=|C|+ |(X\cap Y)\setminus C| \\
&\leq f_{\alpha}(C)+f_{\alpha}((X\cap Y)\setminus C) && \text{(by the $f_{\alpha
}$-sparsity)} \\
&\leq k|V(C)|+k|V((X\cap Y)\setminus C)|-2\ell+\beta((X\cap Y)\setminus C) && \text{(since $C$ is balanced)} \\
&\leq k|V(X\cap Y)|-2\ell+k+\beta((X\cap Y)\setminus C) && \text{(since $C$ is a fraction)} \\
&\leq k|V(X\cap Y)|-2\ell+\min\{2k,\beta(X)+\beta(Y)\},
\end{align*}
where the last inequality follows from  
$\beta((X\cap Y)\setminus C)\leq \min\{\beta(X),\beta(Y)\}=\min\{k,\beta(Y)\}$.
This upper bound of $|X\cap Y|$ contradicts the  lemma assumption,
and (\ref{eq:lem45_1}) follows. 

By  (\ref{eq:lem45_1}),  we can take the above spanning tree $T$  such that $T\cap X_v'=\emptyset$. 
Then by Proposition~\ref{prop:1} there is an element $g\in \Gamma\setminus \{1_{\Gamma}\}$ such that
\begin{equation}
\label{eq:lem45_2}
\text{$\psi(e)=g$ for every $e\in X_v'$
and $\psi(e)=1_{\Gamma}$ for every $e\in X\setminus X_v'$.}
\end{equation} 
%We have two subcases.

If $X'_v\cap Y\neq \emptyset$, then $g\in \psi(Y)$ by (\ref{eq:lem45_2}),
and hence $\alpha(\psi(Y))=\alpha(\psi(Y)\cup \{g\})=\alpha(\psi(X\cup Y))$.
Therefore, we have
\[
\beta(X)+\beta(Y)=\beta(X)+\alpha(\psi(Y))=\beta(X)+\alpha(\psi(X\cup Y))
\geq \beta(X\cap Y)+\beta(X\cup Y),\]
where the first equation follows since $Y$ is not $\alpha$-critical and
the third inequality follows due to the definition of $\beta$.

On the other hand, if $X'_v\cap Y=\emptyset$, 
then $X\cap Y$ is balanced since $\psi(e)=1_{\Gamma}$ for every $e\in X\cap Y$ by (\ref{eq:lem45_2}).
If $Y$ is also balanced, then $\psi(e)=1_{\Gamma}$ for every $e\in Y$, which means that 
$X\cup Y$ is $\alpha$-critical. 
Thus \[\beta(X)+\beta(Y)=\beta(X)=k=\beta(X\cup Y)=\beta(X\cup Y)+\beta(X\cap Y).\]
If $Y$ is unbalanced, then \[\beta(X\cup Y)\leq \tilde{\alpha}(X\cup Y)= \alpha(\psi(X\cup Y))=\alpha(\psi(Y)\cup \{g\}),\] and we get 
\[\beta(X\cup Y)-\beta(Y)\leq \alpha(\psi(Y)\cup \{g\})-\alpha(\psi(Y))\leq k\]
where the first inequality follows since $Y$ is not $\alpha$-critical and the  last inequality follows from  (\ref{eq:assumption}). Therefore,  
\[\beta(X)+\beta(Y)=k+\beta(Y)\geq \beta(X\cup Y)=\beta(X\cap Y)+\beta(X\cup Y),\]
where the last equality follows since $X\cap Y$ is balanced.

\medskip
\noindent
(Case 3)  Suppose that both $X$ and $Y$ are $\alpha$-critical.
If $X\cap Y$ is not $\alpha$-critical, then $X\cap Y$ is balanced by Lemma~\ref{lem:alpha1}.
Since $\beta(X\cup Y)\leq \ell$ and $\beta(Y)= k$, we get
\[\beta(X)-\beta(X\cap Y)= k> \ell-k\geq \beta(X\cup Y)-\beta(Y)\] as required.  
Hence we may assume that $X\cap Y$ is $\alpha$-critical.
Also, by the cardinality assumption for $X\cap Y$ with $\beta(X)+\beta(Y)=2k$, we have that $X\cap Y$ is an $f_{\alpha}$-sparse set with  
$|X\cap Y|\geq k|V(X\cap Y)|-2\ell+2k+1$.
Hence  $X\cap Y$ is $f_{\alpha}$-full. 
Therefore, by Lemma~\ref{lem:tight1}, there is a unique base $v$ for the near-balancedness of $X\cap Y$. Now let $F_X\subseteq X$ and $F_Y\subseteq Y$ be extra edge sets for the near-balancedness of $X$ and the near-balancedness of $Y$, respectively.
Then $F_X\cap X\cap Y$ and $F_Y\cap X\cap Y$ are extra edge sets for the near-balancedness for $X\cap Y$. However, the extra edge set is uniquely determined (up to complementation of non-loop edges) by Lemma~\ref{lem:tight2}, and hence we may assume that $F_Y$ is taken so that $F_X\cap X\cap Y=F_Y\cap X\cap Y$.
Moreover, since $X\cap Y$ has a unique base, 
the bases of $X, Y$ and $X\cap Y$ coincide. 

By Lemma~\ref{lem:near1}, 
$G_X-F_X$, $G_Y-F_Y$, and $G_{X\cap Y}-F_X-F_Y$ are connected,
and we can take the above spanning tree $T$ of $G_{X\cup Y}$
such that $T\cap F_X=\emptyset$ and $T\cap F_Y=\emptyset$.
By  Proposition~\ref{prop:1}, we get
$\psi(e)=g$ for $e\in F_X\cup F_Y$ 
and $\psi(e)=1_{\Gamma}$ for $e\notin F_X\cup F_Y$.
Therefore by Proposition~\ref{prop:0} $X\cup Y$ is near-balanced, 
and moreover it is $\alpha$-critical by $\tilde{\alpha}(X\cup Y)=\alpha(\psi(X\cup Y))=\alpha(\{g\})>k$.
Therefore, we get 
$\beta(X)+\beta(Y)=2k=\beta(X\cup Y)+\beta(X\cap Y)$.
%\geq \alpha(\Gamma_X\cap \Gamma_Y)+k+\alpha(\Gamma_X\cup \Gamma_Y)+k
%\geq \beta(X\cap Y)+\beta(X\cup Y)$.
This completes the proof.
\end{proof}

For $F\subseteq E(G)$, let $d_F=k|V(F)|-|F|$.
Note that, if $G$ is $f_{\alpha}$-sparse, then $d_F\geq \ell-\beta(F)\geq 0$ for every $F\in {\cal C}(G)$.
%, and $F$ is $f_{\alpha}$-tight if and only if $d_F=\ell-\beta(F)$.

\begin{lemma}
\label{lem:main1}
Suppose that $(G,\psi)$ is $f_{\alpha}$-sparse.
Then, for any $f_{\alpha}$-tight sets  $X, Y\in {\cal C}(G)$ with $X\cap Y\neq \emptyset$,
 $X\cup Y$ is $f_{\alpha}$-tight.
\end{lemma}
\begin{proof}
Since $(G,\psi)$ is $f_{\alpha}$-sparse, we have $d_{X\cup Y}\geq \ell-\beta(X\cup Y)$,
and what we have to prove is $d_{X\cup Y}\leq \ell-\beta(X\cup Y)$.
In particular, if $d_{X\cup Y}\leq 0$, then $X\cup Y$ is $f_{\alpha}$-tight.

Let $G_1=(V(X)\cap V(Y), X\cap Y)$.
Let $c_0$ and $c_1$ be the numbers of trivial and non-trivial connected components in $G_1$, where a connected component is said to be trivial if it consists of a single vertex.
Without loss of generality we assume $\beta(X)\geq \beta(Y)$.
Due to the monotonicity of $\beta$, we have $\beta(Y)\geq \beta(F)$ for each edge set $F$ of the connected component of $G_1$.
Hence 
\begin{align}
\nonumber
d_{X\cup Y}&=k|V(X\cup Y)|-|X\cup Y| \\ \nonumber
&=k(|V(X)|+|V(Y)|-|V(X)\cap V(Y)|)-(|X|+|Y|-|X\cap Y|) \\ \label{eq:tight1}
&= d_X+d_Y-kc_0-d_{X\cap Y} \\ \label{eq:lem48_-1} 
&=2\ell-\beta(X)-\beta(Y)-kc_0-d_{X\cap Y} \\ \label{eq:lem48_0}
&\leq 2\ell-\beta(X)-\beta(Y)-kc_0-(\ell-\beta(Y))c_1 \\
&=\ell-\beta(X)-kc_0-(\ell-\beta(Y))(c_1-1). \label{eq:lem48_1}
\end{align}
We first remark the following.
\begin{claim}
\label{claim:1}
If  $d_{X\cup Y}>0$, then 
$|X\cap Y|> k|V(X\cap Y)|-2\ell+\beta(X)+\beta(Y)$, $c_0\leq 1$, and $c_1=1$ hold.
\end{claim}
\begin{proof}
If $|X\cap Y|\leq k|V(X\cap Y)|-2\ell+\beta(X)+\beta(Y)$, then $d_{X\cap Y}=k|V(X\cap Y)|-|X\cap Y|\geq 2\ell-\beta(X)-\beta(Y)$.
Combining this with (\ref{eq:lem48_-1}), we get 
$d_{X\cup Y}\leq -kc_0\leq 0$.

If $c_1\geq 2$, then we have $d_{X\cup Y}\leq 0$ by (\ref{eq:lem48_1}).

If $c_1\leq 1$, then $c_1=1$ holds by $X\cap Y\neq \emptyset$.
Now (\ref{eq:lem48_1}) implies $0\leq d_{X\cup Y}\leq \ell-\beta(X)-kc_0$, and hence 
$kc_0\leq \ell\leq 2k-1$. Therefore $c_0\leq 1$. 
\end{proof}

%Suppose that $d_{X\cup Y}\leq 0$. 
%Since $(G,\psi)$ is $f_{\alpha}$-sparse, we have $0\geq d_{X\cup Y}\geq \ell-\beta(X\cup Y)\geq 0$.
%Thus we get $d_{X\cup Y}=\ell-\beta(X\cup Y)$, and $X\cup Y$ is $f_{\alpha}$-tight. 

As remarked at the beginning of the proof, $d_{X\cup Y}\leq 0$ immediately implies the $f_{\alpha}$-tightness of $X\cup Y$.
Therefore, we may assume $d_{X\cup Y}>0$, and by Claim~\ref{claim:1} we have
$c_0\leq 1$, $c_1=1$, and 
\begin{equation}
\label{eq:lem48_2}
|X\cap Y|> k|V(X\cap Y)|-2\ell+\min\{2k, \beta(X)+\beta(Y)\}. 
\end{equation}
By $c_1=1$, $X\cap Y$ is connected.
We split the proof into two cases depending on the value of $(c_0, c_1)$.

\medskip
\noindent(Case 1) Suppose that $(c_0, c_1)=(0,1)$.
By (\ref{eq:tight1}), we have 
\begin{equation}
\label{eq:tight2}
\begin{split}
\ell-\beta(X\cup Y)\leq d_{X\cup Y}&=  d_X+d_Y-d_{X\cap Y} \\ 
&\leq \ell-\beta(X)-\beta(Y)+\beta(X\cap Y).
\end{split}
\end{equation}
By $(c_0,c_1)=(0,1)$ and (\ref{eq:lem48_2}), we can apply Lemma~\ref{lem:4} to get
$\beta(X)+\beta(Y)\geq \beta(X\cap Y)+\beta(X\cup Y)$.
This means that each inequality holds with equality in (\ref{eq:tight2}),
and in particular we get $d_{X\cup Y}=\ell-\beta(X\cup Y)$.
In other words, $X\cup Y$ is $f_{\alpha}$-tight.

\medskip
\noindent
(Case 2) Suppose that $(c_0, c_1)=(1,1)$.
By  (\ref{eq:tight1}), we have 
\begin{equation}
\label{eq:tight3}
\begin{split}
\ell-\beta(X\cup Y)\leq d_{X\cup Y}&\leq  d_X+d_Y-d_{X\cap Y}-k \\ 
&\leq \ell-\beta(X)-\beta(Y)+\beta(X\cap Y)-k.
\end{split}
\end{equation}
Hence, to prove $d_{X\cup Y}=\ell-\beta(X\cup Y)$,  it suffices to show that 
 \begin{equation}
 \label{eq:tight4}
 \beta(X)+\beta(Y)\geq \beta(X\cup Y)+\beta(X\cap Y)-k.
 \end{equation}
 
 Let $v$ be the vertex isolated in $G_1$, and assume that all edges in $G$ incident to $v$ are directed to $v$.
  Since $(c_0,c_1)=(1,1)$, there is a unique  fraction of $v$ in $G_Y$ whose edge set intersects $X$.  See Figure~\ref{fig:lem48}, and denote the edge set of the fraction by $Y'$.

 We take a spanning tree $T$ of $G_{X\cup Y}$ such that 
 $T\cap X\cap Y$ is a spanning tree of $G_{X\cap Y}$,
 $T\cap X$ is a spanning tree of $G_X$, and $T\cap Y_v'=\emptyset$.
 Let $\psi'$ be a $T$-respecting equivalent gain function
 and let $\Gamma_{Y}=\langle Y \rangle_{u,\psi'}$ for some $u\in V(Y')\setminus \{v\}$.
 Take an edge $e\in Y_v'$ and let $g=\psi'(e)$.
 For each $f\in Y_v'$, there is a closed walk in $(T\cap Y)\cup \{e,f\}$ starting at $u$ and passing through $e$ and $f$ consecutively.
 The gain of this walk is $\psi'(e)\psi'(f)^{-1}$, and hence $\psi'(e)\psi'(f)^{-1} \in \Gamma_Y$.
 This implies 
 \begin{equation}
 \label{eq:lem48_4}
 \text{$\psi'(f)\in \Gamma_{Y} g$ for each $f\in Y_v'$.}
 \end{equation}
 On the other hand, for $f\in Y\setminus (Y'\cup T)$, 
 there is a closed walk in $(T\cap Y)\cup \{e,f\}$ starting at $u$ and passing through $e$, $f$, and then $e$ (in the reversed direction for the last $e$).
 Its gain is $g\psi'(f)g^{-1}$, and we get 
 \begin{equation}
 \label{eq:lem48_5}
 \text{$\psi'(f)\in g^{-1}\Gamma_{Y} g$ for each $f\in Y\setminus Y_v'$.}
 \end{equation}
 Also, since $X\cup Y$ contains a cycle with gain $g$, we have 
 \begin{equation}
 \label{eq:lem48_6}
\text{$\langle X\cup Y\rangle_{u,\psi'}=\langle \psi'(X)\cup \Gamma_{Y}\cup \{g\}\rangle$ in $(G,\psi')$.}
\end{equation} 
 
 \begin{figure}
 \centering
 \begin{minipage}{0.45\textwidth}
 \centering
 \includegraphics[scale=0.8]{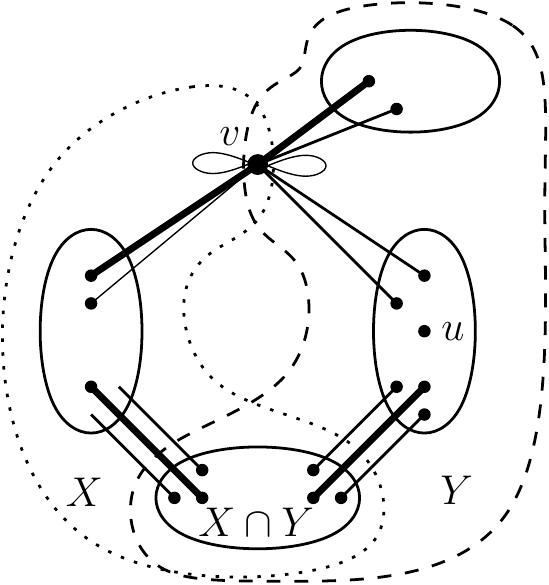}
 \par
 (a)
 \end{minipage}
 \begin{minipage}{0.45\textwidth}
 \centering
 \includegraphics[scale=0.8]{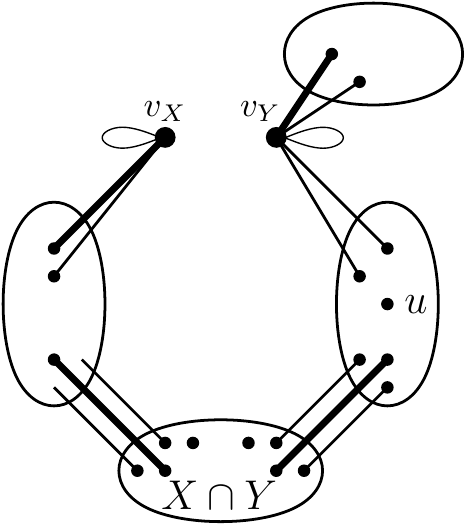}
 \par
 (b)
 \end{minipage}
\caption{Proof of Lemma~\ref{lem:main1}.
(a) $G_{X\cup Y}$, where $G_X$ is the dotted region and $G_Y$ is the dashed region. 
The bold edges represent edges in $T$. 
%The oriented edges are those in $Y_v'$, each of which has a label in $g$. The other edges in $Y$ have labels in $\Gamma_Y$. 
(b) $(H,\psi')$.}
\label{fig:lem48}
\end{figure}
  
 Now to see (\ref{eq:tight4}), we consider $(H,\psi')$ obtained from $(G_{X\cup Y}, \psi')$ by splitting $v$ into two vertices $v_X$ and $v_Y$ such that all edges in $X_v$ are incident to $v_X$ and those in $Y_v$ are incident to $v_Y$. 
 Then $V(X\cap Y)=V(X)\cap V(Y)$ in the resulting graph (Figure~\ref{fig:lem48}(b)),
 and by Lemma~\ref{lem:4} we have 
 $\beta(X)+\beta(Y)\geq \beta(X\cap Y)+\beta(X\cup Y)$ in $(H,\psi')$.
 We now identify the two split  vertices of $H$ to get back $G_{X\cup Y}$.
 Then $\beta(X\cup Y)$ may increase, but we claim that the amount of the increase is bounded by $k$.
 To see this, observe that  $\langle X\cup Y\rangle_{u,\psi'}=\langle \psi'(X)\cup \Gamma_Y\rangle$ in $(H,\psi')$ by (\ref{eq:lem48_4}) and (\ref{eq:lem48_5}).
 On the other hand,  by (\ref{eq:lem48_6}),  $\langle X\cup Y\rangle_{u,\psi'}=\langle \psi'(X)\cup \Gamma_Y\cup \{g\}\rangle$ in $(G_{X\cup Y},\psi')$.
 Therefore, if $\alpha(X\cup Y)$ changes by more than $k$ (i.e., $\alpha(\psi'(X)\cup \Gamma_Y\cup \{g\})-\alpha(\psi'(X)\cup \Gamma_Y)>k$), 
 then $\psi'(X)\cup \Gamma_Y=\{1_{\Gamma}\}$ by (\ref{eq:assumption}).
 This means that $X\cup Y$ is near-balanced in $(G,\psi')$,
 and $\beta(X\cup Y)$ is bounded by $k$ after the identification.
 Hence the increase of the $\beta$-value is bounded by $k$ when identifying the split vertices,
 and we obtain (\ref{eq:tight4}).
% Now $d_{X\cup Y}=\ell-\beta(X\cup Y)$ follows from 
% (\ref{eq:tight3}) and (\ref{eq:tight4}).
% This implies that $X\cup Y$ is $f_{\alpha}$-tight.
\end{proof}

\begin{lemma}
\label{lem:main2}
Let $X\in {\cal C}(G)$ be an $f_{\alpha}$-tight set, 
$Y\in {\cal C}(G)$ be an $f_{\alpha}$-full set, and $e\in E(G)\setminus Y$.
Suppose that  $X\subseteq Y$, $X+e\in {\cal C}(G)$, and $f_{\alpha}(X+e)=f_{\alpha}(X)$.
Then $f_{\alpha}(Y+e)=f_{\alpha}(Y)$.
Moreover $Y+e$ is $f_{\alpha}$-full.
\end{lemma}
\begin{proof}
Since $f_{\alpha}(X+e)=f_{\alpha}(X)$, it can be easily checked that 
both endvertices of $e$ are contained in $V(X)$ and $\beta(X)=\beta(X+e)$.
Thus $|V(Y+e)|=|V(Y)|$, and for  $f_{\alpha}(Y+e)=f_{\alpha}(Y)$ 
it suffices to show that $\beta(Y+e)=\beta(Y)$.
This is trivial if $\beta(Y)=\ell$. So we assume $\beta(Y)<\ell$.

Since the endvertices of $e$ are contained in $V(X)$ and $\beta(X+e)=\beta(X)$,
$X+e$ is $f_{\alpha}$-full.
Moreover, since $X$ is $f_{\alpha}$-tight, 
$|X|=k|V(X)|-\ell+\beta(X)> k|V(X)|-2\ell+\beta(X+e)+\beta(Y)$ 
by $\beta(X+e)=\beta(X)$ and $\beta(Y)<\ell$.
Therefore we can apply Lemma~\ref{lem:4} to get
$0=\beta(X+e)-\beta(X)\geq \beta(Y+e)-\beta(Y)$, 
implying the required relation, $\beta(Y+e)=\beta(Y)$.
This also implies that $Y+e$ is $f_{\alpha}$-full.
\end{proof}

We are now ready to prove Theorem~\ref{thm:main1}.
Our proof also gives an explicit formula for the rank and hence we shall restate it in a different form.
\begin{theorem}
Let $(G,\psi)$ be a $\Gamma$-labeled graph with $G=(V,E)$ and ${\cal I}_{\alpha}$ be the family of 
all $f_{\alpha}$-sparse edge subsets in $E$. 
Then $(E, {\cal I}_{\alpha})$ is a matroid on the ground-set $E$. 
The rank of  the matroid is equal to
\begin{equation*}
\label{eq:rank}
\min\left\{ |E_0|+\sum_{i=1}^t f_{\alpha}(E_i) \ \Bigg| \ E_0\subseteq E, E_i\in {\cal C}(G): 
\{E_0, E_1,\dots, E_t\} \text{ is a partition of } E\right\}.
\end{equation*}
\end{theorem}
\begin{proof}
We say that a partition ${\cal P} = \{E_0, E_1, \dots , E_t\}$ of $E$ is {\em valid} if 
$E_i\in {\cal C}(G)$ for $1\leq i\leq t$.
For a valid partition ${\cal P}$, we denote ${\rm val}({\cal P}) = |E_0|+\sum_{i=1}^t f_{\alpha}(E_i)$. 
We shall check the following independence axiom of matroids: (I1) $\emptyset \in  {\cal I}_{\alpha}$; 
(I2) for any $X,Y \subseteq E$ with $X\subseteq Y$, $Y\in {\cal I}_{\alpha}$ implies $X\in {\cal I}_{\alpha}$; 
(I3) for any $E' \subseteq E$, maximal subsets of $E'$ belonging to ${\cal I}_{\alpha}$ have the same cardinality.

It is obvious that ${\cal I}_{\alpha}$ satisfies (I1).
Also (I2) follows from the definition of the $f_{\alpha}$-sparsity. 
To see (I3), take  a maximal $f_{\alpha}$-sparse subset $F$ of $E$. 
For any valid partition ${\cal P}$, we have $|F| \leq {\rm val}({\cal P})$ 
by $|F| = \sum_{i=0}^t |F\cap E_i|\leq |F\cap E_0|+\sum_{i=1}^t f_{\alpha}(E_i)\leq {\rm val}({\cal P})$.
We shall prove that there is a valid partition ${\cal P}$ of $E$ 
with $|F| = {\rm val}({\cal P})$, from which (I3) follows.

Let $E_0$ be the set of edges which are not contained in any $f_{\alpha}$-tight set in $F$,
and consider the family $\{ F_1, F_2, \dots , F_t\}$ of all inclusion-wise maximal $f_{\alpha}$-tight sets in $F$. 
%Let $J = (V(G),F)$ denote the subgraph of $G$ with the edge set $F$. 
Then $E_0 \cup \bigcup_{i=1}^t F_i=F$ holds. 
Since $F_i \cap  F_j = \emptyset$ for every pair $1 \leq i < j \leq t$ by Lemma~\ref{lem:main1} and the maximality, ${\cal P}_F = \{E_0, F_1,F_2,...,F_t\}$ is a valid partition of $F$ 
and $|F| = {\rm val}({\cal P}_F)$ holds.
%Based on ${\cal P}_F$, we construct a partition ${\cal  P}$ of $E'$ with ${\rm val}({\cal P}) = {\rm val}({\cal P}_F)=|F|$. 

Now consider an edge $e = (u,v) \in E\setminus F$. 
Since $F$ is a maximal $f_{\alpha}$-sparse subset of $E$,  
there is a set $X_e \subseteq F$ with $X_e+e\in {\cal C}(G)$ and  $|X_e+e|>f_{\alpha}(X_e+e)$. 
Let $A=\{e\in E\setminus F: X_e \in {\cal C}(G)\}$ and 
$B=E\setminus (F\cup A)$.

For each $e\in A$, since $X_e$ is $f_{\alpha}$-sparse, 
we have $|X_e|=f_{\alpha}(X_e)= f_{\alpha}(X_e +e)$, 
which implies that $X_e$ is  $f_{\alpha}$-tight and $X_e \subseteq  F_i$ for some $1\leq i\leq t$. 
Choose such an $F_i$ for each $e\in A$ 
and define $E_i =F_i\cup \{e\in A: F_i \text{ was chosen for } e \}$ for $1\leq i\leq t$. 
Then ${\cal P}=\{E_0, E_1,E_2,\dots ,E_t\}$ is a valid partition of $E\setminus B$. 
Moreover, repeated applications of Lemma~\ref{lem:main2} imply 
$f_{\alpha}(F_i)=f_{\alpha}(E_i)$ for every $1 \leq i \leq t$. 
Thus ${\rm val}({\cal P} ) = {\rm val}({\cal P}_F ) = |F|$.

In order to make ${\cal P}$ to a valid partition of $E$, we update ${\cal P}$ by the following process.
Consider any $e\in B$.
Since $X_e+e$ is connected but $X_e$ is not, $e$ is a bridge in $G_{X_e+e}$ and 
$X_e$ can be partitioned into two connected parts 
$X_e^1$ and $X_e^2$.
Due to the $f_{\alpha}$-sparsity, we have 
\begin{equation}
\label{eq:last}
\begin{split}
&k|V(X_e)|-\ell+\beta(X_e+e)=f_{\alpha}(X_e+e)<|X_e+e| \\
&=|X_e^1|+|X_e^2|+1\leq k|V(X_e)|-2\ell+\beta(X_e^1)+\beta(X_e^2)+1,
\end{split}
\end{equation}
implying $\beta(X_e^1)+\beta(X_e^2)\geq \ell+\beta(X_e+e)$.
On the other hand, by the monotonicity of $\beta$, 
$\beta(X_e^1)+\beta(X_e^2)\leq \ell+\beta(X_e+e)$.
Therefore we have $\beta(X_e^1)=\beta(X_e^2)=\beta(X_e+e)=\ell$,
and (\ref{eq:last}) implies that $X_e^1$ and $X_e^2$ are $f_{\alpha}$-tight.
Hence each of $X_e^1$ and $X_e^2$ is contained in some $E_i\in {\cal P}\setminus \{E_0\}$.

If $X_e^1$ and $X_e^2$ are both contained in the same $E_i$, 
then we have $f_{\alpha}(E_i+e)=k|V(E_i+e)|=k|V(E_i)|=f_{\alpha}(E_i)$ by $\ell\geq \beta(E_i)\geq \beta(X_e^1)=\ell$.
Hence we update ${\cal P}$ by replacing $E_i$ with $E_i+e$, which keeps ${\rm val}({\cal P})$.

If $X_e^1$ and $X_e^2$ are not contained in the same $E_i$, 
then without loss of generality assume that $E_i$ contains $X_e^i$ for $i=1,2$.
We have $f_{\alpha}(E_1\cup E_2+e)=k|V(E_1\cup E_2+e)|=k|V(E_1)|+k|V(E_2)|
=f_{\alpha}(E_1)+f_{\alpha}(E_2)$ by $\ell\geq \beta(E_i)\geq \beta(X_e^i)=\ell$ for each $i=1,2$.
Therefore we update ${\cal P}$ by removing $E_1$ and $E_2$ from ${\cal P}$ and inserting $E_1\cup E_2+e$. 
This again keeps ${\rm val}({\cal P})$. 

We perform the above modification one by one for each $e\in B$.
Since each update keeps ${\rm val}({\cal P})$, we finally get a valid partition ${\cal P}$ of $E$ with 
$|F|={\rm val}({\cal P})$.
This completes the proof.
\end{proof}

\section{Checking the Sparsity}
Let $k$ and $\ell$  be two integers with $k\geq 1$ and $0\leq \ell\leq 2k-1$,
and $\alpha$ be a polymatroidal function on $2^{\Gamma}$.
In this section we show how to check  the $f_{\alpha}$-sparsity of a given $\Gamma$-labeled graph $(G,\psi)$ in polynomial time. This also gives an algorithm for checking the independence and computing the rank of the matroid induced by $f_{\alpha}$.
We assume that we are given an oracle that returns $\alpha(X)$ in polynomial time 
for each $X\subseteq \Gamma$.

We first give an algorithm to compute $f_{\alpha}(F)$ for a given $F\in {\cal C}(G)$.
We need to show how to compute $\beta(F)$.
To compute $\tilde{\alpha}(F)$, we fist take any spanning tree $T$ in $G_F$,
and compute the $T$-respecting equivalent $\psi'$ by switching.
Then $\psi'(F)$ generates $\langle F\rangle_{v,\psi'}$ for any $v\in V(F)$ (see, e.g., \cite{jkt} for a detailed exposition), and hence $\tilde{\alpha}(F)=\alpha(\psi'(F))$.
Thus $\tilde{\alpha}(F)$ can be computed in polynomial time.

To compute $\beta(F)$, it remains to check whether $F$ is near-balanced.
For this,  we test whether a vertex $v\in V(F)$ can be a base or not as follows.
We take a spanning tree $T$ of $G_F$ by extending a spanning forest of $G_F-v$,
and let $\psi'$ be a $T$-respecting equivalent gain function.
Proposition~\ref{prop:0} implies that $v$ is a base for the near-balancedness of $F$ if and only if 
$F$ is unbalanced and
there is a non-identity element $g\in \Gamma$ such that 
\begin{itemize}
\item $\psi(e)=1_{\Gamma}$ for  $e\in F\setminus F_v$,
\item  for each fraction $S$ of $G_F$ at $v$, either $\psi(e)\in \{1_{\Gamma}, g\}$ or $\psi(e)\in \{1_{\Gamma}, g^{-1}\}$ for $e\in F_v\cap E(S)$, 
\item $\psi(e)\in \{g, g^{-1}\}$ for every $(L_v\cap F)\setminus L_v^\circ$.
\end{itemize}
%
%If there are a nonidentity element $g\in \Gamma$ and a subset $F_v'$ of $F_v$ such that $\psi'(f)=g$ for $f\in F_v'$ and $\psi'(f)=1_{\Gamma}$ for $f\in F\setminus F_v'$, then $F$ is near-balanced by Proposition~\ref{prop:0} (where $v$ is a base of the near-balancedness).
%Conversely, if $v$ is a base of the near-balancedness, 
%then such a subset  $F_v'$ does exist for some $e\in F_v$  by Proposition~\ref{prop:1}.
%Thus we can conclude that $F$ is not near-balanced if the above test is failed 
%for every $v\in V(F)$ and $e\in F_v$.
Thus one can check whether $v$ can be a base by computing a $T$-respecting equivalent gain function $\psi'$.

For  checking $f_{\alpha}$-sparsity, we need the following simple lemma.
Recall that the $(k,\ell)$-count matroid ${\cal M}_{k,\ell}(G)$ of $G$ consists of the set of all $(k,\ell)$-sparse edge sets in $G$ as  the independent set family.

\begin{lemma}
\label{lem:algo1}
$(G,\psi)$ is $f_{\alpha}$-sparse if and only if 
$G$ is $(k,0)$-sparse and 
$|C|\leq f_{\alpha}(C)$ for every nonempty $C\subseteq E(G)$ 
that is a circuit in ${\cal M}_{k,\ell'}(G)$ for some $1\leq \ell'\leq \ell$.
\end{lemma}
\begin{proof}
The necessity is trivial, and  we prove the sufficiency.
Suppose to the contrary that $(G,\psi)$ is not $f_{\alpha}$-sparse.
Take any $F\in {\cal C}(G)$ such that $|F|>f_{\alpha}(F)$.
Then $|F|>f_{\alpha}(F)\geq k|V(F)|-\ell$.
On the other hand, since $G$ is $(k,0)$-sparse, we have $|F|\leq k|V(F)|$.
Therefore, there is an integer $\ell'$ with $1\leq \ell'\leq \ell$ such that $|F|=k|V(F)|-\ell'+1$.
Since $F$ is dependent in  ${\cal M}_{k,\ell'}(G)$, 
$F$ contains a circuit $C$ in ${\cal M}_{k,\ell'}(G)$.
Note that $k|V(F)|-\ell'-|F|=-1=k|V(C)|-\ell'-|C|$.
Hence by the monotonicity of $\beta$, we get $0\leq f_{\alpha}(C)-|C|\leq f_{\alpha}(F)-|F|<0$, which is a contradiction.
\end{proof}

Based on Lemma~\ref{lem:algo1} we have the following naive algorithm for checking $f_{\alpha}$-sparsity:
\begin{enumerate}
\item Check whether $G$ is $(k,0)$-sparse. If $G$ is not $(k,0)$-sparse, then $(G,\psi)$ is not $f_{\alpha}$-sparse. 
\item For each $\ell'$ with $1\leq \ell'\leq \ell$, enumerate all the circuits in ${\cal M}_{k,\ell'}(G)$
and check wether $|C|\leq f_{\alpha}(C)$ holds for each circuit $C$ in  ${\cal M}_{k,\ell'}(G)$.
If there is a circuit $C$ with $|C|>f_{\alpha}(C)$, then $(G,\psi)$ is not $f_{\alpha}$-sparse;
otherwise it is $f_{\alpha}$-sparse.
\end{enumerate}

It is well-known that checking $(k,0)$-sparsity can be reduced to computing a maximum matching in an auxiliary bipartite graph of size $|V(G)|$, which can be done in $O(|V(G)|^{3/2})$ time (see, e.g., \cite{Frank2011}).
As for the second step, observe that the number of circuits in ${\cal M}_{k,\ell'}(G)$ is $O(|V(G)|^{\ell'-1})$.
This can be seen as follows.
If ${\cal M}_{k,\ell'}(G)$ is not connected (in the matroid sense), then the number of circuits in each connected component $C$ is $O(|V(C)|^{\ell'-1})$ by induction and the sum over all components is $O(|V(G)|^{\ell'-1})$.  
Hence we may assume that ${\cal M}_{k,\ell'}(G)$ is connected,
and the rank of ${\cal M}_{k,\ell'}(G)$ is $k|V(G)|-\ell'$.
Since the size of the ground set is at most $k|V(G)|$ (as $G$ is $(k,0)$-sparse),  
the rank of the dual  of ${\cal M}_{k,\ell'}(G)$ is at most $\ell'$. 
Therefore the number of the hyperplanes in the dual is $O(|V(G)|^{\ell'-1})$, which in turn implies the claimed bound for the number of circuits.

It is known that  all the circuits in a matroid can be enumerated in time polynomial in the size of the ground set and the number of the circuits~\cite{s}, if a polynomial-time oracle for the rank function is available.  
In our case, the number of circuits is polynomial in $|V(G)|$ (assuming that $\ell$ is constant)
and the rank of  ${\cal M}_{k,\ell'}(G)$ can be computed in $O(|V(G)|^2)$ time (see, e.g., \cite{berg,lee}). 
Therefore, the second step can also be done in polynomial time.

Developing a practical polynomial time algorithm whose time complexity is $O(|V(G)|^c)$ for some constant $c$ irrelevant to $\ell$ is left as an open problem. 

\section*{Acknowledgement}
This work was supported by JSPS Postdoctoral Fellowships for Research Abroad, 
JSPS Grant-in-Aid for Young Scientist (B) 24740058, and JSPS Grant-in-Aid for Scientific Research (B) 25280004.


\begin{thebibliography}{99}

\bibitem{berg}
{\scshape A.~Berg and T.~Jord{\'a}n}, 
\newblock Algorithms for graph rigidity and scene analysis, 
\newblock {\it Proc.~11th Annual European Symposium on Algorithms (ESA)}, 
LNCS 2832, (2003), 78--89.



 \bibitem{dowling1973class}
{\scshape T.A. Dowling},
\newblock A class of geometric lattices based on finite groups,
 \newblock {\it  J.~Combin.~Theory Ser.~B}, (1973) \textbf{14}, 61--86.

\bibitem{e70}
{\scshape J.~Edmonds},
\newblock Submodular functions, matroids, and certain polyhedra, 
\newblock {\it Combinatorial Structures and Their Applications}, R. Guy, H. Hanani, N. Sauer, and J. Sch{\"o}nheim, eds., (1970), 69--87.

\bibitem{Frank2011}
{\scshape A.~Frank},
 \newblock Connections in combinatorial optimization, 
 \newblock  Oxford Lecture Series in Mathematics and Its Applications, Oxford University Press, (2011).

\bibitem{gt}
{\scshape J.~L.~Gross and T.~W.~Tucker}, 
\newblock Topological graph theory, 
\newblock Dover, New York, (1987).

\bibitem{ikeshita}
{\scshape R.~Ikeshita},
\newblock Infinitesimal rigidity of symmetric frameworks,
\newblock  Master Thesis, University of Tokyo, (2015).

\bibitem{jkt} 
{\scshape T.~Jord{\'a}n, V.~Kaszanitzky, and S.~Tanigawa},
\newblock  Gain-sparsity and symmetry-forced rigidity in the plane, 
\newblock  {\it Discrete \& Computational Geometry}, (2016) \textbf{55}, 314--372.

\bibitem{lee}
{\scshape A.~Lee and I.~Streinu}, 
\newblock Pebble game algorithms and sparse graphs, 
\newblock {\it Discrete Math.}, (2008) \textbf{308}, 1425--1437.

\bibitem{mt13}
{\scshape J.~Malestein and L.~Theran}, 
\newblock  Generic combinatorial rigidity of periodic frameworks,
\newblock  {\it Adv.~Math.}, (2013) \textbf{233}, 291--331.

\bibitem{mt14}
{\scshape J.~Malestein and L.~Theran}, 
\newblock  Frameworks with forced symmetry II: orientation-preserving crystallographic groups,
\newblock  {\it  Geometriae Dedicata}, (2014) \textbf{170},  219--262.

\bibitem{ns}
{\scshape T.~Nixon and B.~Schulze},
\newblock Symmetry-forced rigidity of frameworks on surfaces,
\newblock {\it Geometriae Dedicata}, (2016) \textbf{182}, 163--201.

\bibitem{ross2011}
{\scshape E.~Ross}, 
\newblock  Geometric and combinatorial rigidity of periodic frameworks as
  graphs on the torus, 
  \newblock Ph.D. thesis, York University, Toronto, (2011).
  
  \bibitem{s}
  {\scshape P.~Seymour}
  \newblock A note on hyperplane generation,
  \newblock {\it J.~Combin.~Theory Ser.~B}, (1994), \textbf{61}, 88-91.
  
\bibitem{st}
{\scshape B.~Schulze and S.~Tanigawa}, 
\newblock  Infinitesimal rigidity of symmetric frameworks,
\newblock  {\it SIAM Discrete Math.}, (2015) \textbf{29}, 1259--1286.

\bibitem{t}
{\scshape S.~Tanigawa},
\newblock  Matroids of gain graphs in applied discrete geometry,
\newblock  {\it Trans.~Amer.~Math.~Soc.}, (2015) \textbf{367}, 8597--8641.

\bibitem{Whitley:1997}
{\scshape W.~Whiteley}, 
\newblock Some matroids from discrete applied geometry,
\newblock {\it  Contemporary Mathematics}, (1996) \textbf{197}, 171--312.

\bibitem{whittle1989generalisation}
{\scshape G.~Whittle},
\newblock A generalisation of the matroid lift construction,
\newblock  {\itshape  Trans.~Amer.~Math.~Soc.}, (1989) \textbf{316}, 141--159.

%\bibitem{zaslavsky1989biased}
%T.~Zaslavsky, \emph{Biased graphs "{I}". bias, balance, and gains},
%  J.~Combin.~Theory Ser.~B \textbf{47} (1989), no.~1, 32--52.

\bibitem{zaslavsky1991biased}
{\scshape T.~Zaslavsky}, 
\newblock Biased graphs "{II}". the three matroids,
\newblock {\itshape J.~Combin.~Theory  Ser.~B}, (1991) \textbf{51}, 46--72.

\bibitem{zaslavsky1994}
{\scshape T.~Zaslavsky}, 
\newblock Frame matroids and biased graphs, 
\newblock {\itshape Eur.~J.~Combin.} (1994) {\bfseries 15}, 303--307
  


%{\scshape A. Author},
%\newblock Title of the paper,
%\newblock  {\itshape Journal } (1999) {\bfseries 52} 


\end{thebibliography}
\end{document}